\documentclass[10pt]{amsart}

\usepackage{amsthm}
\usepackage[ansinew]{inputenc}
\usepackage[english]{babel}
\usepackage{anysize}
\usepackage{amsfonts}
\usepackage{amssymb}
\usepackage{amsmath}
\usepackage{doc}
\usepackage{exscale}
\usepackage{fontenc}
\usepackage{ifthen}
\usepackage{latexsym}
\usepackage{syntonly}
\usepackage{graphicx}
\usepackage{float}
\usepackage{array}
\usepackage[all]{xy}
\usepackage{pstricks}
\usepackage{color}
\usepackage{epsfig}
\usepackage{psfrag}

\newtheorem{thm}{Theorem}
\newtheorem{cor}[thm]{Corollary}
\newtheorem{lem}[thm]{Lemma}
\newtheorem{pro}[thm]{Proposition}
\newtheorem{rmk}[thm]{Remark}
\newtheorem{df}[thm]{Definition}
\newtheorem{alg}[thm]{Algorithm}
\theoremstyle{definition}
\newtheorem{exam}[thm]{Example}
\numberwithin{thm}{section}

\newcommand{\beq}{\begin{equation} }
\newcommand{\enq}{\end{equation}}

\begin{document}

\title{Determining plane curve singularities from its polars}



\author[M. Alberich]{Maria Alberich-Carrami\~nana}
\address{Dept. Matem\`atica Aplicada I\\
Univ. Polit\`ecnica de Catalunya\\ Av. Diagonal 647, Barcelona
08028, Spain}
\address{Institut de Rob\`{o}tica i Inform\`{a}tica Industrial (CSIC-UPC),\\
Llorens i Artigues 4-6, 08028 Barcelona, Spain}
\email{Maria.Alberich@upc.edu}

\author[V. Gonz\'alez]{V\'ictor Gonz\'alez-Alonso}
\address{Institut f\"ur Algebraische Geometrie\\ Leibniz Universit\"at Hannover \\
Welfengarten 1, 30167 Hannover, Germany}
\email{gonzalez@math.uni-hannover.de}

\thanks{This research has been partially supported by the Spanish Committee for Science and Innovation through the I+D+i project M2009-14163-C02-02, and the Catalan Research Commission through the project 2009 SGR 1284. The second author completed this work with the support of grants FPU-AP2008-01849 of the Spanish Ministerio de Educaci\'{o}n and ERC StG 279723 ``Arithmetic of algebraic surfaces'' (SURFARI), which is gratefully acknowledged.}



\maketitle

\begin{abstract}
This paper addresses a very classical topic that goes back at least to Pl\"{u}cker: how to understand a plane curve singularity using its polar curves. Here, we explicitly construct the singular points of a plane curve singularity directly from the weighted cluster of base points of its polars. In particular, we determine the equisingularity class (or topological equivalence class) of a germ of plane curve from the equisingularity class of generic polars and combinatorial data about the non-singular points shared by them.

2010 \emph{Mathematics Subject Classification (MSC2010):} 32S50, 32S15, 14B05.

\end{abstract}

\section{Introduction}

Polar germs are one of the main tools to analyze plane curve singularities, because they carry very deep analytical information on the singularity (see \cite{MY82}). This holds still true for germs of hypersurfaces or even germs of analytic subsets of $\mathbb{C}^{n}$ (see for instance \cite{Tei77b}, \cite{Tei82}, \cite{MY82}, \cite{LT81}, or \cite{Gaf93}). There have been lots of efforts in the literature with the aim of distinguishing which of this information is in fact purely topological. One of the first steps in solving this problem was settled more than thirty years ago by Teissier in \cite{Tei77b}. There, he introduced the polar invariants, which in the planar case can be defined from the intersection multiplicity of the whole curve $\xi$ with the branches of a generic polar, and he proved that they are topological invariants of $\xi$. This result has been generalized by Maugendre in \cite{Mau} and by Michel in \cite{Mi}, where  the role of polars is played by the Jacobian germs of planar morphisms and finite morphisms from normal surface singularities, respectively. The problem of relating a curve to its polars, and vice versa, is the motivation of lots of classical and recent works. Among these let us quote the works of Teissier \cite{Tei77b,Tei82}, Merle \cite{Mer77}, Kuo and Lu \cite{KL77}, L\^{e} and Teissier \cite{LT81}, Eggers \cite{Eggers}, L\^{e}, Michel and Weber \cite{LMW89,LMW91}, Casas-Alvero \cite{Cas90,Cas93}, Gaffney \cite{Gaf93}, Delgado-de la Mata \cite{Del94}, Garc\'\i a-Barroso \cite{Evelia}, and Garc\'\i a-Barroso and Gonz\'alez-P\'erez \cite{Evelia-Pedro}.

In this work we consider the classical topic of understanding a plane curve singularity $\xi$ using its polar curves.
The study of the contact between a reduced plane curve singularity and its polars goes back at least to Pl\"{u}cker, in 1837, in the framework of proving the global projective Pl\"{u}cker formulae \cite{Pluecker}. This motivated later in 1875 the work of Smith \cite{Smith}, which is considered to be the first in giving local results on the contact between a germ of plane curve and its polars.
The question addressed in this paper of determining a plane curve singularity from its polars implies solving two problems. The first one is to choose the right invariant, entirely computable from the polars, which determines the singular points of $\xi$ (or its topological equivalence class), and this was solved by Casas-Alvero in \cite[Theorem 8.6.4]{Cas00}, in the way we will explain next.
The second problem is to explicitly construct the singular points of $\xi$ from this invariant, which is still open and is the scope of this work.

Regarding the first problem, the above mentioned polar invariants are computable from two polar curves taken in different directions (see Lemma \ref{Lem-Rec-Pol-Inv}), or equivalently from the weighted cluster of base points of the Jacobian system, and they could be a starting point. In fact, Merle showed in \cite{Mer77} that for an irreducible $\xi$ the polar invariants and the multiplicity do determine its equisingularity class. However, this does not hold in general and there are examples of reducible non-equisingular curves with the same multiplicity and the same set of polar invariants (see \cite[Example 6.11.7]{Cas00}). Another possibility could be to consider the topological class (or the singular points) of a generic polar, but it turns out that this analytic invariant carries not enough topological information about the singularity. As Casas-Alvero showed in \cite[Theorem 8.6.4]{Cas00}, one has to consider a slightly sharper invariant: the weighted cluster of base points of the polars of $\xi$, which solves the first problem. Indeed, the underlying cluster consists of the singular points of the generic polars plus the non-singular points shared by generic polars (or by all polars, if we are considering the notion of ``going virtually through a cluster'' of infinitely near points, as it will be explained in Section \ref{ssect-inf}).

The second problem of giving the singular points of $\xi$ from its polars is still open. In fact, Casas-Alvero's proof of Theorem 8.6.4 in \cite{Cas00} is highly non-constructive, and nothing is said about the relation between both objects. Only for an irreducible $\xi$ the answer follows easily from the explicit formulas given by Merle in \cite{Mer77}.


The aim of this work is to present an algorithm which explicitly recovers the weighted cluster of singular points of a plane curve singularity directly from the base points of its polar germs. Recognizing the difference in difficulty, this could be interpreted as a sort of local version of the known, quite elementary fact in algebraic geometry that the proper singular points of plane projective algebraic curves are exactly the proper base points of its polar curves. In particular, the algorithm applies to describe the equisingularity class of a germ of plane curve (by giving this information combinatorially encoded by means of an Enriques diagram) from the Enriques diagram which encodes the equisingularity class of a generic polar enlarged by some extra vertices representing the simple (non-singular) points shared by generic polars. As we will show, these extra vertices are only relevant for recovering the polar invariants. Once the polar invariants are computed as a previous step in Lemma \ref{Lem-Rec-Pol-Inv}, our procedure shows in which way the equisingularity class (or the singular points) of generic polars determines the equisingularity class of the curves. Furthermore, our approach applies for any pair of polars in different directions, regardless whether they are topologically generic or even transverse ones (see Corollary \ref{Cor-4.11}). As an additional value, our algorithm gives a quite clear and neatly different proof of Casas-Alvero's Theorem 8.6.4 of \cite{Cas00}. We address the problem by reinterpreting it in terms of the theory of planar analytic morphisms, recently developed in \cite{Cas07}, and a careful and ingenious use of these new techniques enables us to construct our new proof.


Falling on the same stream of recovering the equisingularity class of a germ of plane curve from invariants associated to polars, but starting form a different setting, there are the works by Eggers and by Garc\'{\i}a-Barroso.
In \cite{Eggers}, Eggers proves that the generic polar enriched with the polar invariants corresponding to each of its branches determine the equisingularity (topological) type of the curve. Hence the starting data include some information about the topological type of the curve, and it is crucial to know which polar invariant correspond to each branch of the polar, since the permutation of two polar invariants may give different topological types of curve (as shown in \cite{Eggers} or \cite{Evelia}).
In \cite[Theorem 6.1]{Evelia}, Garc\'{\i}a-Barroso proves that the {\em partial polar invariants} of a plane curve $\xi$ and the multiplicities of its branches determine the equisingularity type of the curve. Partial polar invariants are defined from the intersection multiplicity of each branch of $\xi$ with the branches of a generic polar. Hence, in order to have the partial polar invariants at the beginning, one needs to know some information about the topological type of $\xi$ (the number of branches, their multiplicity, and their intersection with each branch of the polar). Our work, instead, does not take for granted any knowledge of the original curve $\xi$, and its equisingularity type is computed entirely from the polars.


This paper is structured as follows. In Section \ref{sect-prel} we give a survey on the tools used all along the work, recalling definitions and facts about infinitely near points, polar germs of singular curves and germs of planar morphisms. We then relate our problem about polar germs to the theory of planar analytic morphisms and close the section with a short sketch of the algorithm giving the solution. Section \ref{sect-tracking} contains the technical results needed to solve the problem, which we believe are interesting on their own. It is divided into two parts. The first part is devoted to the study of the growth of some rational invariants, $I_{\xi}(p)$, associated to the equisingularity class of the curve, independently of its polars. The behaviour of these invariants has been studied by several authors, but always considering only points $p$ lying on $\xi$. However, we need to take into account also points which do not lie on the curve, as well as some refined versions of the known results for points of the curve. Therefore, we have developed some generalizations that, although not particularly surprising, are new and essential for our work. The second part studies the relation between these topological invariants, the values $v_p(\xi)$ of the curve and some invariants, the multiplicities $n_p$ and the heights $m_p$, of the morphism associated to a generic polar. Finally, in Section \ref{sect_rec} we develop the results which build up our algorithm and apply it to a paradigmatic example of Pham and to a more complicated curve with several branches, some of them with more than one characteristic exponent, illustrating how the algorithm works.

\vspace{3mm}
\noindent \textbf{Acknowledgement}
The authors thank F. Dachs-Cadefau for the implementation of the algorithm.

\section{Preliminaries and translation of the problem to a morphism}

\label{sect-prel}

In this section we introduce the notations and concepts needed in the development of the results of this work. We start recalling some notions about infinitely near points, equisingularity of plane germs of curve and base points of linear systems, followed by some results relating them to polar germs. Next we expose a brief review of the theory of planar analytic morphisms developed by Casas-Alvero in \cite{Cas07}, explaining  how our problem fits in that context. The last part of the section is a short overview of the main ideas behind our algorithm to solve the problem. For the sake of brevity, we have kept this section merely descriptive, and the reader is referred, for instance, to \cite[Chapters 3, 4 and 6]{Cas00} and \cite{Cas07} for further details or proofs.

\subsection{Infinitely near points.}

\label{ssect-inf}

From now on, suppose $O$ is a smooth point in a complex surface $S$, and denote by $\mathcal{O} = \mathcal{O}_{S,O}$ the local ring at $O$, i.e. the ring of germs of holomorphic functions in a neighbourhood of $O$. We denote by $\mathcal{N}_O$ the set of points {\em infinitely near to} $O$ (including $O$), which can be viewed as the disjoint union of all exceptional divisors obtained by successive blowing-ups above $O$. The points in $S$ will be called {\em proper} points in order to distinguish them from the infinitely near ones. Given any $p \in \mathcal{N}_O$, we denote by $\pi_p: S_p \longrightarrow S$ the minimal composition of blowing-ups that realizes $p$ as a proper point in a surface $S_p$, and by $E_p$ the exceptional divisor obtained by blowing up $p$ in $S_p$, which is also called its {\em first neighbourhood}. The set $\mathcal{N}_O$ is naturally endowed with and order relation $\leqslant$ defined by $p \leqslant q$ (resp. $p < q$, reading $p$ {\em precedes} $q$) if and only if $q \in \mathcal{N}_p$ (resp. $q \in \mathcal{N}_p - \{ p \}$).

A function $f \in \mathcal{O}$ defines a {\em (germ of) curve} $\xi: f = 0$ at $O$, whose {\em branches} are the germs given by the irreducible factors of $f$. The germ $\xi$ is irreducible if and only if its equation is irreducible. In the sequel, we will implicitly assume that all the curves are {\em reduced} (i.e. they have no multiple branches). The {\em multiplicity} of $\xi$ at $O$, $e_O(\xi)$, is defined to be the order of vanishing of the equation $f$ at $O$. From now on consider that $\xi: f = 0$ is a given curve at $O$. For any $p \in \mathcal{N}_O$ we denote by $\bar{\xi}_p: \pi_p^*f = 0$ its {\em total transform} at $p$, which contains a multiple of the exceptional divisor of $\pi_p$. If we subtract these components we obtain the {\em strict transform} $\widetilde{\xi}_p$, which might be viewed as the closure of $\pi_p^{-1}(\xi - \{O\})$. The {\em multiplicity} and the {\em value} of $\xi$ at $p$ are defined respectively as $e_p(\xi) = e_p(\widetilde{\xi}_p)$ and $v_p(\xi) = e_p(\bar{\xi}_p)$.
We say that $p$ {\em lies on} $\xi$ if and only if $e_p(\xi) > 0$, and we denote by $\mathcal{N}_O(\xi)$ the set of all such points. A point $p \in \mathcal{N}_O(\xi)$ is {\em simple} (resp. {\em multiple}) if and only if $e_p(\xi) = 1$ (resp. $e_p(\xi) > 1$). In the case $\xi$ is irreducible, $\mathcal{N}_O(\xi)$ is totally ordered and the sequence of multiplicities is non-increasing.

Given two germs of curve $\xi,\zeta$ without common components, its intersection multiplicity  at $O$ can be computed by means of {\em Noether's formula} (see \cite[Theorem 3.3.1]{Cas00}) as
\begin{equation} \label{Noether}
[\xi.\zeta]_O = \sum_{p \in \mathcal{N}_O(\xi) \cap \mathcal{N}_O(\zeta)} e_p(\xi) e_p(\zeta).
\end{equation}

Given $p \leqslant q$ points infinitely near to $O$, $q$ is {\em proximate} to $p$ (written $q \rightarrow p$) if and only if $q$ lies on the exceptional divisor $E_p$. A point $p$ is {\em free} (resp. {\em satellite}) if it is proximate to exactly one point (resp. two points), and these are the only possibilities. Note that $q \rightarrow p$ implies $q \geqslant p$, but not conversely.

\begin{df} \label{df_SatelliteOf}
We say that $q$ is {\em satellite of} $p$ (or {\em $p$-satellite}) if $q$ is satellite and $p$ is the last free point preceding $q$ (cf. \cite[Section 3.6]{Cas00}).
\end{df}

Proximity allows to establish the {\em proximity equalities}
\begin{equation} \label{proximity}
e_p(\xi) = \sum_{q \rightarrow p} e_q(\xi),
\end{equation}
and the following relation between values and multiplicities
\begin{equation} \label{mult-vals}
v_p(\xi) = e_p(\xi) + \sum_{p \rightarrow q} v_q(\xi).
\end{equation}

A point $p \in \mathcal{N}_O(\xi)$ is {\em singular} (on $\xi$) if it is either multiple, or satellite, or precedes a satellite point $q \in \mathcal{N}_O(\xi)$. Equivalently, $p \in \mathcal{N}_O(\xi)$ is non-singular if and only if it is free and there is no satellite point $q \in \mathcal{N}_O(\xi)$, $q > p$. The set of singular points of $\xi$ weighted by the multiplicities or the values of $\xi$ at them is denoted by $\mathcal{S}(\xi)$. Two curves $\xi, \zeta$ are {\em equisingular} if it exists a bijection $\varphi: \mathcal{S}(\xi) \longrightarrow \mathcal{S}(\zeta)$ (called an {\em equisingularity}) preserving the natural order $\leqslant$, the multiplicities (or values) and the proximity relations. It is known that two such curves are equisingular if and only if they are topologically equivalent in a neighbourhood of $O$ (seen as germs of topological subspaces of $\mathbb{C}^{2}=\mathbb{R}^{4}$). Thus, $\mathcal{S}(\xi)$ determines the topological class of (the embedding of) the curve $\xi$.

The set of singular points of a curve is a special case of a (weighted) cluster. A {\em cluster} is a finite subset $K \subset \mathcal{N}_O$ such that if $p \in K$, then any other point $q < p$ also belongs to $K$. A {\em weighted cluster} $\mathcal{K}=(K,\nu)$ is a cluster $K$ together with a function $\nu: K \longrightarrow \mathbb{Z}$. The number $\nu_p = \nu(p)$ is the {\em virtual multiplicity} of $p$ in $\mathcal{K}$. Two clusters $K,K'$ are {\em similar} if there exists a bijection ({\em similarity}) $\varphi: K \longrightarrow K'$ preserving the ordering and the proximity. In the weighted case we also impose $\varphi$ to preserve the virtual multiplicities.

A cluster can be represented by means of an {\em Enriques diagram} (\cite{EC85a,EC85b}), which is a rooted tree whose vertices are identified with the points in $K$ (the root corresponds to the origin $O$) and there is an edge between $p$ and $q$ if and only if $p$ lies on the first neighbourhood of $q$ or vice-versa. Moreover, the edges are drawn according to the following rules:
\begin{itemize}
\item If $q$ is free, proximate to $p$, the edge joining $p$ and $q$ is curved and if $p \neq O$, it is tangent to the edge ending at $p$.
\item If $p$ and $q$ ($q$ in the first neighbourhood or $p$) have been represented, the rest of points proximate to $p$ in successive neighbourhoods of $q$ are represented on a straight half-line starting at $q$ and orthogonal to the edge ending at $q$.
\end{itemize}
In the weighted case, the vertices are labeled with their virtual multiplicities.

Another usual way to represent a cluster $K$ is the {\em dual graph} of the exceptional divisor of $\pi_K: S_K \longrightarrow S$, the composition of the successive blow-ups of every point in $K$. It is another tree, which has one vertex corresponding to each exceptional curve of $\pi_K$ (and hence, to each point $p \in K$), and two vertices are joint by an edge if and only if the corresponding exceptional curves intersect in $S_K$. It is naturally rooted at the vertex corresponding to $O$, and the choice of this root induces a partial ordering $\prec$ in $K$ (different than the natural ordering $\leq$) that later plays an important role.

Both the Enriques diagram and the dual graph may be used to represent the equisingularity class of a curve $\xi$. One starts with the representation of $\mathcal{S}\left(\xi\right)$, and then one add an edge for each branch $\gamma$ of $\xi$, starting at the vertex corresponding to the last singular point on $\gamma$ and without end. In the Enriques diagram these edges are curved, and in the dual graph they are usually arrows (pointing out of the graph). We will call these graphs {\em augmented} Enriques diagram or dual graph.

A curve $\xi$ goes through $O$ with virtual multiplicity $\nu_O$ if $e_O(\xi) \geq \nu_O$, and in this case the {\em virtual transform} is $\check{\xi} = \bar{\xi} - \nu_O E_O$. This definition can be extended inductively to any point $p \in K$ whenever the multiplicities of the successive virtual transforms are non-smaller than the virtual ones. In this case it is said that $\xi$ {\em goes (virtually) through} the weighted cluster $\mathcal{K}$. If moreover $e_p(\xi) = \nu_p$ for all $p \in K$, it is said that $\xi$ goes through $\mathcal{K}$ {\em with effective multiplicities equal to the virtual ones}. It might happen that there is no curve going through a given weighted cluster with effective multiplicities equal to the virtual ones, but when there exists such a curve the cluster is said to be {\em consistent}. Furthermore, if this is the case, there are curves going through $\mathcal{K}$ with effective multiplicities equal to the virtual ones and missing any finite set of points not in $K$. Equivalently, $\mathcal{K}$ is consistent if and only if $\nu_p \geq \sum_{q \rightarrow p} \nu_q$ for all $p \in K$, which resembles the proximity equalities (\ref{proximity}). In this case, the difference $\rho_p = \nu_p - \sum_{q \rightarrow p} \nu_q$ is the {\em excess} of $\mathcal{K}$ at $p$, and $p$ is {\em dicritical} if and only if $\rho_p > 0$. Finally, we say that $\xi$ goes {\em sharply} through $\mathcal{K}$ if it goes through $\mathcal{K}$ with effective multiplicities equal to the virtual ones and furthermore it has no singular points outside $K$. All germs going sharply through a consistent cluster are reduced and equisingular (cf. \cite[Proposition 4.2.6]{Cas00}), or more generally, germs going sharply through similar consistent clusters are equisingular. Moreover, if $\xi$ goes sharply through $\mathcal{K}$ and $p \in K$, $\xi$ has exactly $\rho_p$ branches going through $p$ and whose point in the first neighbourhood of $p$ is free and does not belong to $K$.

\begin{df} \label{df_Kp}
Given $p \in \mathcal{N}_O$, we denote by $\mathcal{K}(p)$ the (irreducible weighted cluster) consisting of the points $q \leq p$ such that $\rho_p = \nu_p = 1$ and $\rho_q = 0$ for every $q < p$. Thus, germs going sharply through  $\mathcal{K}(p)$ are irreducible, with multiplicity one at $p$, and its (only) point in the first neighbourhood of $p$ is free and non-singular.
\end{df}

Based on Noether's formula, it is possible to define the intersection number of a weighted cluster with a curve, or even two clusters, as
\[[\mathcal{K}.\xi] = [\xi.\mathcal{K}] = \sum_{p \in K} \nu_p e_p(\xi) \qquad
\mbox{ and } \qquad [\mathcal{K}.\mathcal{K}'] = \sum_{p \in K \cap K'} \nu_p \nu_p'.\]
In particular, the self-intersection of a weighted cluster is defined as $\mathcal{K}^2 = \sum_{p \in K} \nu_p^2.$

The main example of weighted cluster is the cluster $BP\left(\mathcal{L}\right)$ of base points of a linear family $\mathcal{L}$ of curves without fixed part (i.e., the curves in $\mathcal{L}$ have no common component). It has multiplicity $\nu_O = \min\{e_O(\xi) \, | \, \xi \in \mathcal{L}\}$ at the origin, and the multiplicities at the infinitely near points are computed inductively considering the virtual transforms of $\xi \in \mathcal{L}$. All germs in $\mathcal{L}$ go virtually through $BP(\mathcal{L})$, and generic ones go sharply through it, miss any fixed finite set of points not in $BP(\mathcal{L})$, and in particular are reduced and have the same equisingularity class. In the particular case $\mathcal{L}$ is a pencil, any two such germs share exactly the points in $BP(\mathcal{L})$, and the self-intersection $BP(\mathcal{L})^2$ coincides with the intersection of two distinct germs in $\mathcal{L}$.

\subsection{Polar germs and its base points.}

In this section we remind the basic definitions and facts about polar germs of curve. We will assume $\xi: f = 0$ is a non-empty, reduced, singular germ of curve at $O$. A {\em polar} of $\xi$ is any germ given by the vanishing of the jacobian determinant
\begin{equation}
\label{def-polar} P_g(f): \frac{\partial(f,g)}{\partial(x,y)} = \left|
\begin{array}{cc}
\frac{\partial f}{\partial x} & \frac{\partial f}{\partial y} \\
\frac{\partial g}{\partial x} & \frac{\partial g}{\partial y}
\end{array}
\right| = 0
\end{equation}
with respect to some local coordinates $(x,y)$ at $O$, where $g$ defines  a smooth germ $\eta$ at $O$. The equation (\ref{def-polar}) actually defines a curve unless $\xi$ is a multiple of $\eta$ (in this case the determinant vanishes identically), which we assume not to hold from now on. We might even suppose that $\eta$ is not a component of $\xi$, since in this case the polar is composed by $\eta$ and the polar of $\xi - \eta$.
A polar is {\em transverse} if the curve $\eta$ is not tangent to $\xi$. The set of polar curves obtained in this way does not depend on the choice of coordinates (\cite[Remark 6.1.1]{Cas00}), but it does actually depend on the equation $f$, and not only on the curve $\xi$ itself (\cite[Remark 6.1.6]{Cas00}). However, this is not a problem because we are interested in intrinsical properties of the polar curves depending only on $\xi$, namely properties of its {\em jacobian ideal}, defined as ${\bf J}(\xi) = \left(f,\frac{\partial f}{\partial x},\frac{\partial f}{\partial y}\right) \subset \mathcal{O}$. This ideal does not depend on the choice of the equation $f$ for $\xi$, and carries very deep information about the singularity of $\xi$. Indeed, it was shown by Mather and Yau in \cite{MY82} that two germs $\xi_1,\xi_2$ are {\em analytically equivalent} if and only if the rings $\mathcal{O}/{\bf J}(\xi_1)$ and $\mathcal{O}/{\bf J}(\xi_2)$ are isomorphic.

The jacobian ideal defines a linear system $\mathcal{J}(\xi)$ called the {\em jacobian system} of $\xi$. Although all the polars belong to the jacobian system, the converse is not true. However, every germ in the jacobian system of multiplicity $e_O(\xi)-1$ is indeed a polar curve. If $\xi$ is reduced and singular, its jacobian ideal is not the whole ring $\mathcal{O}$, its jacobian system is without fixed part, and hence its generic members are reduced and go sharply through its weighted cluster of base points $BP(\mathcal{J}(\xi))$ (hence they are equisingular and, furthermore, they share all their singular points). This motivates the following

\begin{df}
Let $\zeta$ be a polar of a reduced singular curve $\xi$. We say that $\zeta$ {\em is topologically generic} if it goes sharply through $BP(\mathcal{J}(\xi))$.
\end{df}

The weighted cluster $BP(\mathcal{J}(\xi))$ is difficult to compute from its definition, but it can be shown (cf. \cite{Tei77a} and \cite[Corollary 8.5.7]{Cas00}) that it coincides with $BP\left(\frac{\partial f}{\partial x},\frac{\partial f}{\partial y}\right)$, the weighted cluster of base points of the pencil spanned by the partial derivatives of any equation of $\xi$. But base points of pencils are easy to compute (see for instance the algorithm in \cite{Alb2004-CommAlg}).

The cluster $BP(\mathcal{J}(\xi))$ is deeply related to the cluster of singular points of $\xi$. As a first result, it contains all the free singular points of $\xi$ (\cite[Lemma 8.6.3]{Cas00}), but the most striking result is the following

\begin{thm}(\cite[Theorem 8.6.4]{Cas00}) \label{Casas-8.6.4}
Let $\xi_1$ and $\xi_2$ be germs of curve, both reduced and singular. Then
\begin{enumerate}
\item If $BP(\mathcal{J}(\xi_1)) = BP(\mathcal{J}(\xi_2))$, then $\mathcal{S}(\xi_1) = \mathcal{S}(\xi_2)$.
\item If $BP(\mathcal{J}(\xi_1))$ and $BP(\mathcal{J}(\xi_2))$ are similar weighted clusters, then $\xi_1$ and $\xi_2$ are equisingular.
\end{enumerate}
\end{thm}

The proof of Casas-Alvero works in two steps. The first one is to recover the polar invariants (which will be introduced below), and the second step is a procedure involving a careful tracking of the Newton polygon of the iterated strict transforms of a generic polar under blowing up. However, the major drawback of this proof is that it throws no light on the connection between the singular points of both objects: germ of curve and generic polars.

Our aim is to give a precise description of the relation between the singular points of the curve and those of its generic polars. This will provide a new alternative proof of Theorem \ref{Casas-8.6.4}. As a previous step we will also recover the polar invariants, but in contrast, our algorithm will give a different proof of the second step, avoiding the use of the Newton polygon and the tracking of the polars after successive blowing-ups.

A classical tool to study the relation between a germ and its polar curves are the polar invariants. These invariants were introduced by Teissier in \cite{Tei77b}, where he proved that they are topological invariants of $\xi$ closely related to its (transverse) polar curves. A point $p \in \mathcal{N}_O(\xi)$ is a {\em rupture point} of $\xi$ if either there are at least two free points on $\xi$ in its first neighbourhood, or $p$ is satellite and there is at least one free point on $\xi$ in its first neighbourhood. Equivalently, $p$ is a rupture point if and only if the total transform $\bar{\xi}_p$ has three different tangents. In the augmented dual graph of $\mathcal{S}\left(\xi\right)$, rupture points correspond to vertices with three or more incident edges (counting the arrows). We denote by $\mathcal{R}(\xi)$ the set of rupture points of $\xi$. More generally, if $p \in \mathcal{N}_O$ is a free point, $\mathcal{R}^p(\xi)$ denotes the subset of rupture points of $\xi$ which are either equal to $p$ or $p$-satellite. Note that all rupture points are singular, and also all maximal singular points are rupture points.

For any $p \in \mathcal{N}_O$, take $\gamma^p$ to be any irreducible germ of curve going through $p$ and whose point in the first neighbourhood of $p$ is free and does not lie on $\xi$, and define the rational number \begin{equation} \label{def-rational-inv} I(p) = I_{\xi}(p) = \frac{[\xi.\gamma^p]}{e_O(\gamma^p)} = \frac{[\xi.\mathcal{K}(p)]}{\nu_O(\mathcal{K}(p))}, \end{equation} which is independent of the choice  of $\gamma^p$ and will be called \emph{invariant quotient at} $p$. The {\em polar invariants} of $\xi$ are the invariant quotients $I(q)$ at the rupture points $q \in \mathcal{R}(\xi)$. Note that they (as well as the invariant quotients) can be computed from an Enriques diagram of $\xi$, and hence are topological invariants of $\xi$. In fact, it was shown by Merle in \cite{Mer77} that if $\xi$ is irreducible, its equisingularity class is determined by its multiplicity at $O$ and by its polar invariants. Polar invariants have an interesting topological meaning which was given by L\^{e}, Michel and Weber in \cite{LMW91}.

We have defined the polar invariants without any mention to polar germs. Its relation to polar germs is given by the next

\begin{pro}(\cite[Theorems 6.11.5 and 6.11.8]{Cas00}) \label{pro-polar-inv}
Let $\zeta = P_g(\xi)$ be a transverse polar of a non-empty reduced germ of curve $\xi$, and let $\gamma_1, \ldots, \gamma_l$ be the branches of $\zeta$. Then \[\left\{\frac{[\xi.\gamma_i]}{e_O(\gamma_i)}\right\}_{i=1,\ldots,l} = \left\{I(q)\right\}_{q \in \mathcal{R}(\xi)}.\] Furthermore, if $p \in \mathcal{N}_O(\xi)$ is either $O$ or any free point lying on $\xi$, the set of quotients $\frac{[\xi.\gamma]}{e_O(\gamma)}$, for $\gamma$ a branch of $\zeta$ going through $p$ and missing all free points on $\xi$ after $p$, is just $\left\{I(q)\right\}_{q \in \mathcal{R}^p(\xi)}$.
\end{pro}

\subsection{Planar analytic morphisms.}

We end the preliminary material summarizing some definitions and results concerning germs of morphisms between surfaces which will be used along the paper. We now consider two points $O \in S$, $O' \in T$ lying on two smooth surfaces. A {\em germ of morphism} of surfaces at them is a morphism $\varphi: U \longrightarrow V$ defined on some neighbourhoods of $O$ and $O'$, such that $\varphi(O) = O'$. We will assume that the morphism is dominant, i.e. its image is not contained in any curve through $O'$, or equivalently the pull-back morphism $\varphi^*: \mathcal{O}_{T,O'} \longrightarrow \mathcal{O}_{S,O}$ is a monomorphism. Since the surfaces are smooth, we can attach two systems of coordinates $(x,y)$ and $(u,v)$ centered at $O$ and $O'$ respectively, obtaining isomorphisms $\mathcal{O}_{S,O} \cong \mathbb{C}\{x,y\}$ and $\mathcal{O}_{T,O'} \cong \mathbb{C}\{u,v\}$. Under this isomorphisms, we denote by $\hat{h} \in \mathbb{C}[x,y]$ the {\em initial form} of any $h \in \mathcal{O}_{S,O}$, and by $o_O(h) = \deg \hat{h}$ its order (and analogously for $h' \in \mathcal{O}_{T,O'}$).

The pull-back of germs at $O'$ is defined by pulling back equations,  and the push-forward, or direct image, of germs at $O$ is defined on irreducible germs and then extended by linearity. For an irreducible germ $\gamma$ at $O$ its push-forward $\varphi_*(\gamma)$ is defined as the image curve $\sigma = \varphi(\gamma)$ counted with multiplicity equal to the degree of the restriction $\varphi_{\gamma}: \gamma \rightarrow \sigma$. With this definitions, it holds the {\em projection formula}
\begin{equation} \label{proj-form}
[\xi.\varphi^*(\zeta)]_O = [\varphi_*(\xi).\zeta]_{O'}
\end{equation}
for all germs of curve $\xi$ at $O$ and $\zeta$ at $O'$.

Let $(f(x,y),g(x,y))$ be the expression of $\varphi$ in  the coordinates fixed above. The {\em multiplicity} of $\varphi$ is defined as $e_O(\varphi) = n = n_O = \min\{o_O(f),o_O(g)\}$. Consider now the pencil $\mathcal{P} = \{\lambda f + \mu g = 0\}$. Its fixed part $\Phi$ is the {\em contracted germ} of $\varphi$, defined by $h = \gcd(f,g)$. If both $\frac{f}{h}$ and $\frac{g}{h}$ are non-invertible, the variable part $\mathcal{P}'$ is a pencil without fixed part whose cluster of base points is by definition the {\em cluster of base points} of $\varphi$, denoted $BP(\varphi)$. The {\em multiplicity} $e_p(\varphi)$ of $\varphi$ at any point $p \in \mathcal{N}_O$ infinitely near to $O$ is defined as the sum of $e_p(\Phi)$ and the virtual multiplicity of $BP(\varphi)$ at $p$. A point $p$ is {\em fundamental} of $\varphi$ if $e_p(\varphi) > 0$. The multiplicity can alternatively be extended to any $p \in \mathcal{N}_O$ as the multiplicity of the composition $\varphi_p = \varphi \circ \pi_p$, which is denoted by $e(\varphi_p)$ or $n_p$ if the morphism is clear from the context. These two possible generalizations of the notion of multiplicity correspond respectively to the multiplicities and the values of a curve at a point. Indeed, they verify the following formula (see \cite[Proposition 13.1]{Cas07})
\begin{equation}
\label{mult-val-morph} e(\varphi_p) = e_p(\varphi) + \sum_{p \rightarrow q} e(\varphi_q).
\end{equation}

So far we have attached to $\varphi$ a weighted cluster of points infinitely near to $O$. There is a natural way to construct a weighted cluster of points at $O'$: the trunk of $\varphi$. Let $\mathcal{L} = \{l_{\alpha} : \alpha \in \mathbb{P}_{\mathbb{C}}^1\}$ be a pencil of lines at $O$, and consider its direct images $\{\gamma_{\alpha} = \varphi_*(l_{\alpha})\}$. All but finitely many of them may be parametrized as
\[(u(t),v(t)) = (t^n, \sum_{i \geq n} a_i t^i)\]
where $n = e_O(\varphi)$ and the $a_i$ may depend on $\alpha$. Indeed, since $\varphi$ is supposed to be dominant, at least one of them will depend on $\alpha$. Since the coefficients of a Puiseux series determine the position of the points (cf. \cite[Chapter 5]{Cas00}), all but finitely many of the $\gamma_{\alpha}$ share a finite number of points with the same multiplicities. This weighted cluster is independent of the choice of the pencil of lines $\mathcal{L}$, it is denoted by $\mathcal{T} = \mathcal{T}(\varphi)$, and it is called the {\em (main) trunk} of $\varphi$. The smallest integer $m = m_O$ such that $a_m$ is not constant is the {\em height} of the trunk. These definitions can be extended to any $p \in \mathcal{N}_O$ by considering the morphism $\varphi_p$ instead of $\varphi$. In \cite[Section 10]{Cas07} it is developed an algorithm to compute the trunk of any morphism from its expression in coordinates.

The last concept we want to recall is the {\em jacobian germ} or $\varphi$. It is defined as the germ
\[{\bf J}(\varphi) : \frac{\partial(f,g)}{\partial(x,y)} = \left|
\begin{array}{cc}
\frac{\partial f}{\partial x} & \frac{\partial f}{\partial y} \\
\frac{\partial g}{\partial x} & \frac{\partial g}{\partial y}
\end{array}
\right| = 0,\]
which is a germ of curve at $O$ (the determinant does not vanish  identically because $\varphi$ is dominant). Note that when $g$ defines a smooth germ, the jacobian germ is a polar of $\xi: f = 0$. One of the main results of \cite{Cas07} gives an explicit formula to compute the multiplicities of the jacobian germ from the multiplicities and the heights of the trunks of the composites $\varphi_p$:

\begin{pro}(\cite[Theorem 14.1]{Cas07}) \label{pro-mult-jac}
For any point $p \in \mathcal{N}$, we have
\begin{equation} \label{form-mult-jac}
e_p({\bf J}(\varphi)) =
\begin{cases}
m+n-2 & \mbox{ if $p = O$,} \\
m_p+n_p-m_{p'}-n_{p'}-1 & \mbox{ if $p$ is free, proximate to $p'$,} \\
m_p+n_p-m_{p'}-n_{p'}-m_{p''}-n_{p''} & \mbox{ if $p$ is satellite,  prox. to $p'$ and $p''$.}
\end{cases}
\end{equation}
\end{pro}

In particular, we will use the following

\begin{cor}(\cite[Corollary 14.4]{Cas07}) \label{Assian-14.4}
If $p$ is a non-fundamental point of $\varphi$,  then $m_p = m_{p'} + e_p({\bf J}(\varphi)) + 1$ if $p$ is free proximate to $p'$, and $m_p = m_{p'} + m_{p''} + e_p({\bf J}(\varphi))$ if $p$ is satellite proximate to $p'$ and $p''$. In any case, $m_p > m_{p'}$.
\end{cor}

\subsection{The problem.}

Our aim is to give an explicit algorithm which computes the weighted cluster $\mathcal{S}(\xi)$ of singular points of a singular and reduced germ of curve $\xi$ from the weighted cluster of base points of the jacobian system $BP(\mathcal{J}(\xi))$. In particular, we shall obtain a new proof of Theorem \ref{Casas-8.6.4}. To achieve this, we reinterpret the problem in terms of the theory of planar analytic morphisms as follows.

Let $(x,y)$ be a system of coordinates in a neighbourhood  $U$ of $O$, $f$ an equation for the germ $\xi$, and $\eta: g = 0$ a smooth germ at $O$ such that the point on $\eta$ in the first neighbourhood of $O$ is not in $BP(\mathcal{J}(\xi))$ and $\zeta = P_g(f): \frac{\partial(f,g)}{\partial(x,y)} = 0$ is a topologically generic transverse polar of $\xi$. Note that being topologically generic is a generic property, and being transverse excludes finitely many tangent directions at $O$, so the existence of such a $\eta$ is guaranteed.

The key observation is that we can think of the polar $\zeta$ as the jacobian germ of the morphism $\varphi: U \longrightarrow \mathbb{C}^2$ defined as $\varphi(x,y) = (f(x,y),g(x,y))$.

Let us first study the fundamental points of $\varphi$. Since we are assuming $\zeta$ to be transverse, we know that $f$ and $g$ share no factors, so $\varphi$ has no contracted germ. Thus the only fundamental points of $\varphi$ are its base points $BP(\varphi) = BP(\{\xi_{\lambda}: \lambda_1 f + \lambda_2 g = 0\})$. Note that $\xi_{[1,0]} = \xi$ and $\xi_{[0,1]} = \eta$. We have $e_O(\xi_{\lambda}) = 1$ for $\lambda \neq [1,0]$, and so $\nu_O(BP(\varphi)) = 1$. Since the weighted cluster of base points of a pencil is consistent, this forces $BP(\varphi)$ to be irreducible and to have only free points with virtual multiplicity one. Moreover, its self-intersection is $BP(\varphi)^2 = e_O(\xi)$, so $BP(\varphi)$ consists of $e_O(\xi)$ points lying on $\eta$. We have thus proved the following

\begin{lem} \label{Lem-BP(morphism)}
The fundamental points of $\varphi$ are exactly the first $e_O(\xi)$ points in $\mathcal{N}_O(\eta)$. In particular, there are no fundamental points in $BP(\mathcal{J}(\xi))$ but the origin $O$.
\end{lem}

Combining this result with formula (\ref{mult-val-morph})  and Corollary \ref{Assian-14.4} we obtain the following
	
\begin{lem} \label{calcul-n_p-m_p}
If $p \neq O$ is either a base point of $\mathcal{J}(\xi)$ or a satellite of one of them (or more generally, it is not a fundamental point of $\varphi$), then
\begin{eqnarray*}
& & n_p = \sum_{p \rightarrow q} n_q, \\
& & m_p = m_{p'} + e_p(\zeta) + 1 \mbox{ if $p$ is free, proximate to $p'$, and }\\
& & m_p = m_{p'} + m_{p''} + e_p(\zeta) \mbox{ if $p$ is satellite, proximate to $p'$ and $p''$,}
\end{eqnarray*}
while for $p = O$ we have $n_O = 1$ and $m_O = e_O(\xi) = e_O(\zeta) + 1$.
\end{lem}

\subsection{The solution: an algorithm.}

The algorithm we have found to solve this problem is based on the following facts. Firstly, $BP\left(\mathcal{J}\left(\xi\right)\right)$ coincides with the weighted cluster of base points of the pencil $\left(\frac{\partial f}{\partial x},\frac{\partial f}{\partial y}\right)$ spanned by any two polars along different directions. This allows to recover the set of polar invariants just from $BP\left(\mathcal{J}\left(\xi\right)\right)$ (see Lemma \ref{Lem-Rec-Pol-Inv}). Secondly, although the underlying cluster of $BP\left(\mathcal{J}\left(\xi\right)\right)$ does not coincide with the set of singular points of $\xi$, each dicritical point $d$ of $BP\left(\mathcal{J}\left(\xi\right)\right)$ corresponds to a unique rupture point $q_d \in \mathcal{R}\left(\xi\right)$ whose associated polar invariant is given by any branch of a polar going through $d$ (Proposition \ref{Prop-p-gamma}). Furthermore, any rupture point of $\xi$ can be obtained in this way, so $BP\left(\mathcal{J}\left(\xi\right)\right)$ is enough to determine the {\em set} of singular points of $\xi$ (because every maximal singular point is a rupture point). Finally, since the set of singular points does not determine the equisingularity point of the curve (because there are many ways to assign virtual multiplicities in a consistent way), it is necessary to determine the multiplicities of $\mathcal{S}\left(\xi\right)$.

Our algorithm works then roughly as follows (see Algorithm \ref{alg-1} for a precise description). In the first part, for each dicritical point $d$ of $BP\left(\mathcal{J}\left(\xi\right)\right)$ we compute the associated polar invariant $I_d$ and explicitly find the rupture point $q_d$ by comparing $I_d$ with the quotients $\frac{m_p}{n_p}$. In the second part, after finding all rupture and singular points, we determine the values of $\xi$ at any singular point (which indeed coincide with $m_p$ for many $p$, for example for the rupture points). This is clearly equivalent to recover the virtual multiplicities of $\mathcal{S}\left(\xi\right)$ by means of the formula (\ref{mult-vals}).

Algorithm \ref{alg-1} implemented in the Computer Algebra system {\tt Macaulay 2} \cite{GS} will be available at the web page {\tt www.pagines.ma1.upc.edu/$\sim$alberich} or upon request to authors.

\section{Tracking the behaviour of the invariant quotients} \label{sect-tracking}

In this section we develop the main technical results which describe de behaviour of the invariant quotients $I_{\xi}(p)$ as $p$ ranges over $\mathcal{N}_O$, as well as its relation to the values $v_p(\xi)$ and the heights $m_p$ associated to the morphism $\varphi$ introduced at the end of section \ref{sect-prel}.

\subsection{Growth of the invariant quotients}

First of all, we need to introduce a new order relation in $\mathcal{N}_O$. Recall (Definition \ref{df_Kp}) that for any $p \in \mathcal{N}_O$, $\mathcal{K}(p)$ is the irreducible weighted cluster whose last point is $p$.

\begin{df}
Let $q_1 \neq q_2$ be two points infinitely near to $O$, equal to or satellite of the free points $p_1$ and $p_2$ respectively. We say that $q_1$ is {\em smaller than} $q_2$ (or $q_2$ is {\em bigger than} $q_1$), and denote it $q_1 \prec q_2$ (or $q_2 \succ q_1$) if $p_1 \leqslant p_2$ (with the usual order) and $\frac{\nu_{p_1}(\mathcal{K}(q_1))}{\nu_O(\mathcal{K}(q_1))} \leq \frac{\nu_{p_1}(\mathcal{K}(q_2))}{\nu_O(\mathcal{K}(q_2))}$. Obviously, we denote by $q_1 \preceq q_2$ the situation in which $q_1 \prec q_2$ or $q_1 = q_2$, and similarly for $q_2 \succeq q_1$.
\end{df}

We introduce also the following relation between points and irreducible curves.

\begin{df}
Let $\gamma$ be any irreducible germ, let $p$ be any free point and let $q$ be either $p$ or a $p$-satellite point. We say that $q$ is {\em smaller} {\em than} $\gamma$ (or that $\gamma$ is {\em bigger} than $q$) if $p \in \mathcal{N}_O(\gamma)$ (or equivalently $e_p(\gamma)>0$) and $\frac{\nu_p(\mathcal{K}(q))}{\nu_O(\mathcal{K}(q))} < \frac{e_p(\gamma)}{e_O(\gamma)}$. We denote it $q \prec \gamma$.
\end{df}

\begin{rmk}
It is worth noting that the ordering $\prec$ coincides with the ordering in the dual graph. More precisely, if $\Gamma$ is the dual graph of a cluster containing $q_1$ and $q_2$, then $q_1 \prec q_2$ if and only if the vertex corresponding to $q_1$ belongs to the minimal path from $O$ to $q_2$.
\end{rmk}

The following lemmas summarize the main properties of the order relation $\prec$.

\begin{lem} \label{lem-Ordre-p-satellites}
Let $p \in \mathcal{N}_O$ be any free point different from $O$, proximate to $p'$. Then:
\begin{enumerate}
\item The satellite point $q$ in the first neighbourhood of $p$ satisfies $p' \prec q \prec p$.
\item If $q$ is a $p$-satellite point, the two satellite points $q_1, q_2$ in its first neighbourhood may be ordered as $p' \prec q_1 \prec q \prec q_2 \prec p$. Moreover, every $p$-satellite point $q'$ infinitely near to $q_1$ (resp. $q_2$) satisfies $q' \prec q$ (resp. $q' \succ q$).
\end{enumerate}
\end{lem}
\begin{proof}
The proof follows easily from the relation between the set of $p$-satellite points in $\mathcal{K}(q)$ and the expansion as a continued fraction of the quotient $\frac{\nu_p(\mathcal{K}(q))}{\nu_{p'}(\mathcal{K}(q))}$, combined with some elementary properties of continued fractions (see for instance \cite[Remark 2.1 and Lemma 3.5]{AAG10}). Alternatively, the result follows immediately from the fact that $\prec$ coincides with the order in the dual graph.
\end{proof}


For future reference, the point $q_1$ (resp. $q_2$) in the second case above will be called {\em first} (resp. {\em second}) {\em satellite} of $q$. In the first case, when there is only one satellite point $q$, it will be called {\em first satellite} of $p$.

It is also useful to know how a satellite point is ordered with respect to the two points which it is proximate to.

\begin{lem} \label{lem-Ordre-proximate}
Let $q$ be a satellite point, proximate to $q_1$ and $q_2$, and assume $q_1 \prec q_2$. Then
\[q_1 \prec q \prec q_2.\]
\end{lem}

We now turn to the relation between the ordering $\prec$ and the growth of the invariant quotients $I_{\xi}(p)$.

\begin{pro} \label{pro-growing}
Let $p \neq O$ be a free point proximate to $p'$, let $q_1$ be a $p$-satellite point and let $q_2 \succ q_1$ be either $p$ or another $p$-satellite point. Then the following inequalities hold:
\[I_{\xi}(p') \stackrel{(a)}{\leqslant} I_{\xi}(q_1) \stackrel{(b)}{\leqslant} I_{\xi}(q_2).\]
Moreover, equality holds in $(a)$ if and only if $p \not \in \mathcal{N}_O(\xi)$, and equality holds in $(b)$ if and only if there is no branch $\gamma$ of $\xi$ such that $q_1 \prec \gamma$ (bigger than $q_1$). In particular, note that equality in $(a)$ implies equality in $(b)$.
\end{pro}
\begin{proof}
For any infinitely near point $q \in \mathcal{N}_O$, let $\gamma^q$ be any irreducible  curve going through $q$ and having a free point in its first neighbourhood which does not lie on $\xi$. The first inequality, as well as the characterization of equality, is easily obtained computing the intersections $[\xi.\gamma^{p'}]$ and $[\xi.\gamma^{q_1}]$ with Noether's Formula (\ref{Noether})
.

For the second inequality, let $\xi_1, \ldots, \xi_k$ be the branches of $\xi$ and expand each $I_{\xi}(q_i)$ as
\begin{equation} \label{eq2}
I_{\xi}(q_i) = \frac{[\xi.\gamma^{q_i}]}{e_O(\gamma^{q_i})} = \sum_{j=1}^k \frac{[\xi_j.\gamma^{q_i}]}{e_O(\gamma^{q_i})}.
\end{equation}
For branches $\xi_j$ not going through $p$, we have $\frac{[\xi_j.\gamma^{q_1}]}{e_O(\gamma^{q_1})} = \frac{[\xi_j.\gamma^{q_2}]}{e_O(\gamma^{q_2})}$ again by Noether's Formula
. For the rest of the branches, following \cite[Proposition 2.5]{AAG10} we can write
\begin{equation} \label{eq3}
\frac{[\xi_j.\gamma^{q_i}]}{e_O(\gamma^{q_i})} = \sum_{q < p} \frac{e_{q}(\xi_j)^2}{e_O(\xi_j)} + e_{p'}(\xi_j) \min \left\{ \frac{e_p(\xi_j)}{e_O(\xi_j)}, \frac{e_p(\gamma^{q_i})}{e_O(\gamma^{q_i})} \right\},
\end{equation}
and so we just need to take care of the minimum in the last summand. In the case $\frac{e_p(\xi_j)}{e_O(\xi_j)} \leqslant \frac{e_p(\gamma^{q_1})}{e_O(\gamma^{q_1})}$ this minimum is the same for $i=1,2$, while in the opposite case (i.e. when $\xi_j \succ q_1$) the minimum for $i=1$ is strictly smaller than for $i=2$, giving strict inequality in $(b)$ as wanted.
\end{proof}

\begin{rmk}
Proposition \ref{pro-growing} can be interpreted as follows: the function $I_{\xi}$ is monotone increasing on the dual graph of any composition of blow-ups, and strictly increasing over the dual graph of any subset of $\mathcal{N}_O\left(\xi\right)$.
\end{rmk}



Also next corollary follows immediately.

\begin{cor} \label{cor-I-rupt}
If $p \in \mathcal{N}_O(\xi)$ is a free point, all the polar invariants $I_{\xi}(q)$ associated to points $q \in \mathcal{R}^p(\xi)$ are different.
\end{cor}

Unfortunately, these results are not precise enough for our purposes, so we need a more sophisticated result which deals with a particular case.

\begin{pro} \label{pro-I-satellite-3}
Let $\xi$ be a germ of curve at $O$, and $p \in \mathcal{N}_O$ any free point different from $O$. Assume $\xi$ has at least two branches going through $p$, and that exactly one of them, say $\gamma$, goes through a free point in the first neighbourhood of $p$. Suppose in addition that $p$ is a non-singular point of $\gamma$. Finally, let $q \in \mathcal{N}_O(\xi)$ be the biggest $p$-satellite rupture point on $\xi$. Then
\[\frac{[\xi.\mathcal{K}(p)]-1}{\nu_O(\mathcal{K}(p))} = I_{\xi}(p)-\frac{1}{\nu_O(\mathcal{K}(p))} < I_{\xi}(q) < I_{\xi}(p).\]
\end{pro}
\begin{proof}
The second inequality is given by 
Proposition \ref{pro-growing}, so we just need to prove the first one. If we consider decompositions as in (\ref{eq2}) both for $I_{\xi}(p)$ and $I_{\xi}(q)$, the proof of Proposition \ref{pro-growing} shows that all the summands are equal but for the one corresponding to the branch $\gamma$, and that it only remains to check one of the following (equivalent) inequalities
\begin{equation} \label{eq7}
\frac{[\gamma.\gamma^p]}{e_O(\gamma^p)} - \frac{1}{e_O(\gamma^p)} < \frac{[\gamma.\gamma^q]}{e_O(\gamma^q)}, \; \text{or} \; \frac{1}{e_O(\gamma^p)} > \frac{[\gamma.\gamma^p]}{e_O(\gamma^p)} - \frac{[\gamma.\gamma^q]}{e_O(\gamma^q)} = e_{p'}(\gamma) \left(\frac{e_p(\gamma^p)}{e_O(\gamma^p)} - \frac{e_p(\gamma^q)}{e_O(\gamma^q)}\right),
\end{equation}
where $\gamma^p$ and $\gamma^q$ are as in the proof of Proposition \ref{pro-growing}, $\gamma^p$ {\em going sharply} through $\mathcal{K}(p)$, $p'$ is the point $p$ is proximate to, and the last equality is a consequence of \cite[Proposition 2.5]{AAG10}. Now, noting that both $\gamma$ and $\gamma^p$ go sharply through $\mathcal{K}(p)$, we get $e_p(\gamma^p) = e_p(\gamma) = e_{p'}(\gamma) = 1$ and the inequality in (\ref{eq7}) becomes obvious.
\end{proof}

\begin{rmk}
The hypotheses of Proposition \ref{pro-I-satellite-3} can be expressed in terms of the dual graph of $\mathcal{S}\left(\xi\right)$ as follows: $p$ correspond to a maximal vertex, and there is exactly one arrow coming out from it. The point $q$ corresponds to the last rupture vertex in the path from $O$ to $p$.
\end{rmk}

\begin{rmk}
Propositions \ref{pro-growing} and  \ref{pro-I-satellite-3} are generalizations of \cite[Proposition 7.6.8]{Cas00}, extending it to points not necessarily lying on $\xi$ and giving more precise descriptions of some cases. Similar results can be found also in \cite{LMW89}, \cite{CDG99} and \cite{DGN08}.
\end{rmk}

\subsection{Relating the invariant quotients to the morphism}

\label{sect-morph}


We now wish to study the relation between the invariant quotients $I_{\xi}(p)$, the values $v_p(\xi)$ of $\xi$ and the multiplicities $n_p$ and heights $m_p$ of the morphisms $\varphi_p = \varphi \circ \pi_p$ for the points in $BP(\mathcal{J}(\xi))$ or satellite of them (or more generally, for any $p \in \mathcal{N}_O$ such that $\mathcal{K}(p) \cap \mathcal{N}_O(\eta) = \{O\}$).
We begin with an easy

\begin{lem} \label{alternativa-I(p)}
If $p \in \mathcal{N}_O$ belongs to $BP(\mathcal{J}(\xi))$ or is satellite of such a base point (or more generally, $\mathcal{K}(p) \cap \mathcal{N}_O(\eta) = \{O\}$), then $[\xi.\mathcal{K}(p)] = v_p(\xi)$ and $\nu_O(\mathcal{K}(p)) = v_p(\eta) = n_p$. In particular
\[I_{\xi}(p) = \frac{v_p(\xi)}{n_p}.\]
\end{lem}
\begin{proof}
The intersection number $[\xi.\mathcal{K}(p)]$ equals $[\xi.\gamma^p]$ for any $\gamma^p$ going sharply through $p$ and missing any point on $\xi$ in the first neighbourhood of $p$, and this intersection turns out to be $v_p(\xi)$. Indeed, if $\pi_p: S_p \longrightarrow S$ is the composition of blowing-ups giving rise to $p$, then $\gamma^p = \pi_{p*}(l_p)$ for some smooth curve $l_p$ at $p$ non-tangent to $\bar{\xi}_p$. Then, by the projection formula (\ref{proj-form}), we have
\[[\xi.\gamma^p] = [\xi.\pi_{p*}(l_p)] = [\pi_p^*(\xi).l_p] = [\bar{\xi}_p.l_p] = e_p(\bar{\xi}_p) e_p(l_p) = v_p(\xi).\]
For the second part, the virtual multiplicity $\nu_O(\mathcal{K}(p))$ may be written as the intersection $[\eta.\mathcal{K}(p)]$ (because $\mathcal{K}(p) \cap \mathcal{N}_O(\eta) = \{O\}$), and thus $\nu_O(\mathcal{K}(p)) = v_p(\eta)$ by the same reason as above. But the values $v_p(\eta)$ also satisfy the recursive formula of Lemma \ref{calcul-n_p-m_p} with the same initial value $n_O = 1 = e_O(\eta)$, and hence $e_O(\gamma^p) = n_p e_p(\gamma^p)$. The last equality is immediate.
\end{proof}

We now focus on the relation between the values and the heights.

\begin{pro} \label{pro-valors-alcades}
Keeping the hypothesis of Lemma \ref{alternativa-I(p)}, the inequality $v_p(\xi) \leq m_p$ holds, with equality if and only if the total transforms $\bar{\xi}_p$ and $\bar{\eta}_p$ at $p$ have non-homothetical tangent cones (counting multiplicities, or equivalently, considered as divisors on $E_p$, the first neighbourhood of $p$).
\end{pro}
\begin{proof}
The proof is based on the algorithm given in \cite[Section 10]{Cas07} to compute the trunk of a morphism. This algorithm produces a sequence of pencils whose clusters of base points have strictly increasing heights (the definition of the height of a trunk works for any multiple of an irreducible cluster). It is immediate to check that the cluster in the first step of this algorithm has height exactly $v_p(\xi) = o(\varphi_p^*(u))$, and that the algorithm stops after this first step if and only if the initial forms of $\varphi_p^*(u)$ and $\varphi_p^*(v)$ are non-homothetical, which is equivalent to the total transforms $\bar{\xi}_p$ and $\bar{\eta}_p$ at $p$ having non-homothetical tangent cones.
\end{proof}

We are now ready to state the main results relating the values and
the heights:

\begin{thm} \label{Thm-values-versus-alcades}
Still keeping the hypothesis of the previous results, let $p' \leq p$ be the last free point preceding (or equal to) $p$. Then $v_p(\xi) \leq m_p$, with equality if and only if
\begin{itemize}
\item either $p$ is free and there is a free point proximate to $p$ lying on $\xi$ (in particular, $p$ lies on $\xi$),
\item or $p$ is satellite and there exists a branch of $\xi$ which goes through $p'$, and this branch is not smaller than $p$.
\end{itemize}
Equivalently, $v_p(\xi) < m_p$ if and only if all branches of $\xi$ going through $p'$ are smaller than $p$.
\end{thm}
\begin{proof}
Let us first consider the case $p$ free. By Proposition \ref{pro-valors-alcades}, we know that $v_p(\xi) = m_p$ if and only if the total transforms $\bar{\xi}_p$ and $\bar{\eta}_p$ have non-homothetical tangent cones. Since $p$ is free, it is proximate to a single point $q$. Let $E_q$ be the germ (at $p$) of the exceptional divisor of $\pi_p: S_p \longrightarrow S$. By definition, $\bar{\xi}_p = v_q(\xi) E_q + \tilde{\xi}_p$, and by the hypothesis on $p$, $\bar{\eta}_p = n_q E_q$. So, $\bar{\xi}_p$ and $\bar{\eta}_p$ have homothetical tangent cones if and only if every branch of $\tilde{\xi}_p$ is also tangent to $E_q$, which means that there is no free point in the first neighbourhood of $p$ lying on $\xi$. So, $v_p(\xi) = m_p$ if and only if there is some free point in the first neighbourhood of $p$ lying on $\xi$, as wanted.

Now let us deal with the case $p$ satellite, proximate to two points $q$ and $q'$. Assume that $q \prec q'$, so that $q \prec p \prec q'$ by Lemma \ref{lem-Ordre-proximate}. By definition and the hypothesis on $p$ we have $\bar{\xi}_p = v_q(\xi) E_q + v_{q'}(\xi) E_{q'} + \tilde{\xi}_p$ and $\bar{\eta}_p = n_q E_q + n_{q'} E_{q'}$. Let $a_q$ (resp. $a_{q'}$) denote the multiplicity of $E_q$ (resp. $E_{q'}$) in the tangent cone of $\tilde{\xi}_p$. Then $\bar{\xi}_p$ and $\bar{\eta}_p$ have homothetical tangent cones if and only if every branch of $\tilde{\xi}_p$ is tangent to either $E_q$ or $E_{q'}$ (equivalently, $a_q + a_{q'} = e_p(\xi)$) and
\[\frac{v_q(\xi)+a_q}{n_q} = \frac{v_{q'}(\xi)+a_{q'}}{n_{q'}}.\]

So assume $\bar{\xi}_p$ and $\bar{\eta}_p$ have homothetical tangent cones, which by the previous Proposition means that $v_p(\xi) < m_p$, and take $\alpha = \frac{v_q(\xi)+a_q}{n_q} = \frac{v_{q'}(\xi)+a_{q'}}{n_{q'}}$. Then on the one hand we have
\[\alpha = \frac{v_q(\xi)+a_q+v_{q'}(\xi)+a_{q'}}{n_q+n_{q'}} = \frac{v_q(\xi)+v_{q'}(\xi)+e_p(\xi)}{n_p} = \frac{v_p(\xi)}{n_p} = I_{\xi}(p),\]
and on the other hand
\[\alpha = I_{\xi}(q)+\frac{a_q}{n_q} \geqslant I_{\xi}(q) \qquad \text{and} \qquad \alpha = I_{\xi}(q')+\frac{a_{q'}}{n_{q'}} \geqslant I_{\xi}(q').\]
But we have assumed $q \prec p \prec q'$, and thus by Proposition \ref{pro-growing} we have $I_{\xi}(q) \leqslant I_{\xi}(p) \leqslant I_{\xi}(q')$, which combined with the above equalities implies that $I_{\xi}(p) = I_{\xi}(q') \, (=\alpha)$ and $a_{q'} = 0$. This in turn implies (by Proposition \ref{pro-growing}) that every branch of $\xi$ going through $p'$ is smaller than $p$, as wanted, and that $a_q = e_p(\xi)$.

It remains to prove that if $\bar{\xi}_p$ and $\bar{\eta}_p$ have non-homothetical tangent cones (i.e. $v_p(\xi) = m_p$), then there is some branch of $\xi$ going through $p'$ which is not smaller than $p$. But this case only may occur either if $a_q+a_{q'} < e_p(\xi)$ or if $\frac{v_q(\xi)+a_q}{n_q} \neq \frac{v_{q'}(\xi)+a_{q'}}{n_{q'}}.$ In the former case there is a branch of $\xi$ through $p$ whose point in its first neighbourhood is free, and such a branch is not smaller than $p$. In the latter case we can assume that $a_q+a_{q'} = e_p(\xi)$ (for if not we are in the previous case) and then we have that the quotient $I_{\xi}(p) = \frac{v_q(\xi)+a_q+v_{q'}(\xi)+a_{q'}}{n_q+n_{q'}}$ fits between $I_{\xi}(q)+\frac{a_q}{n_q} = \frac{v_q(\xi)+a_q}{n_q}$ and $I_{\xi}(q')+\frac{a_{q'}}{n_{q'}} = \frac{v_{q'}(\xi)+a_{q'}}{n_{q'}}$. Since $p \prec q'$ implies $I_{\xi}(p) \leqslant I_{\xi}(q')$, we are in fact in the situation
\[I_{\xi}(q)+\frac{a_q}{n_q} < I_{\xi}(p) < I_{\xi}(q')+\frac{a_{q'}}{n_{q'}}.\]
Now we have to consider the cases when the second inequality holds. If we already have $I_{\xi}(p) < I_{\xi}(q')$, then by Proposition \ref{pro-growing} there exists a branch of $\xi$ going through $p'$ and bigger than $p$, as we want. If otherwise $I_{\xi}(p) = I_{\xi}(q')$, then $a_{q'} > 0$ and there is at least one branch of $\xi$ whose strict transform at $p$ is tangent to $E_{q'}$. This concludes de proof because this branch is bigger than $p$.
\end{proof}

\begin{cor} \label{Cor-rupt-valor=alcada} If $p$ is a rupture point of $\xi$, then
\[v_p(\xi) = m_p.\]
\end{cor}
\begin{proof}
Since $p$ is a rupture point of $\xi$, there is at least one branch of $\xi$ going through it and whose point in the first neighbourhood if free. Such a branch clearly goes through the last free point preceding or equal to $p$, and is not smaller than $p$. Thus, we have $v_p(\xi) = m_p$ in virtue of Theorem \ref{Thm-values-versus-alcades}.
\end{proof}

\section{Recovering the singular points from the base points of the polars} \label{sect_rec}

This section presents the main result of this paper, namely the procedure which recovers the weighted cluster of singular points $\mathcal{S}(\xi)$ (of a singular reduced germ of curve $\xi$) directly from the weighted cluster $BP(\mathcal{J}(\xi))$ of base points of the jacobian system of $\xi$. This procedure uses only invariants computable from the Enriques diagram of $BP(\mathcal{J}(\xi))$ (weighted with the virtual multiplicities) and hence one of the strengths of this procedure is that it applies also to obtain the topological class of $\xi$ directly from the similarity class of $BP(\mathcal{J}(\xi))$.


\subsection{Recovering rupture points} \label{ssect_rec_rupt}

In order to recover the set of rupture points $\mathcal{R}(\xi)$, and hence the whole set of singular points of $\xi$, just from $BP(\mathcal{J}(\xi))$, we argue as follows. Let $\mathcal{D}$ be the set of dicritical points of $BP(\mathcal{J}(\xi))$. We will show that to each $d \in \mathcal{D}$ we can associate a uniquely determined rupture point $q_d \in \mathcal{R}(\xi)$ such that $I_{\xi}\left(q_d\right) = I_{\xi}\left(d\right)$. Moreover we will see that any rupture point is associated to some dicritical point in this way (see Proposition \ref{Prop-p-gamma}). However, the explicit determination of $q_d$ has two main difficulties to be overcome. On one side, despite the polar invariants $\{I_{\xi}(d)\}_{d \in \mathcal{D}}$ are computable from $BP(\mathcal{J}(\xi))$ (see Lemma \ref{Lem-Rec-Pol-Inv}), it is not possible to know the invariant quotient $I_{\xi}(p)$ for whatever $p$, and hence the possibility to check equality $I_{\xi}(p)=I_{\xi}(d)$ (necessary to identify the rupture point $q_d$ associated to $d$) is out of reach. On the other side, if $q_d$ happens to be $p_d$-satellite, then $q_d$ does not necessarily belong to $BP(\mathcal{J}(\xi))$. Furthermore, despite we manage to characterize the free point $p_d$ 
in terms of the invariants $n_{p_d}$ and $m_{p_d}$ (see Proposition \ref{Prop-p}), there might be many $p_d$-satellite points $q$ with the same invariant quotient $I_{\xi}(q) = I_{\xi}(q_d)$, and some criterion to distinguish $q_d$ must be found. All these difficulties are solved by a cunning use of the invariants $I_{\xi}(q)$, $n_q$ and $m_q$, and their properties developed in Section \ref{sect-morph}. More precisely, as we will exhibit, not only the quotients $\frac{m_q}{n_q}$ behave similarly enough like the invariant quotients $I_{\xi}\left(q\right)$ to help find $p_d$, but they are at the same time different enough to distinguish between the $p_d$-satellite points $q$ when the invariants $I_{\xi}(q)$ cannot (see Theorem \ref{Thm-Calcul-q_gamma}).

Next we will develop the results that justify our procedure, which will be detailed as an algorithm at the end of the section.

\begin{pro} \label{Prop-p-gamma}
Let $d \in \mathcal{D}$ be a dicritical point of $BP(\mathcal{J}(\xi))$, and suppose $p_d$ is the last free point lying both on $\xi$ and $\mathcal{K}(d)$. Then there exists a unique rupture point $q_d \in \mathcal{R}^{p_d}(\xi)$ such that $I_{\xi}(q_d)  = I_{\xi}(d)$. Furthermore, $q_d \preceq d$.

Moreover, any rupture point is associated to some dicritical point in this way.
\end{pro}
\begin{proof}
Let $\gamma$ be a branch of a topologically generic transverse polar $\zeta$ of $\xi$ going sharply through $\mathcal{K}(d)$ (such a $\gamma$ exists because $d$ is a dicritical point of $BP(\mathcal{J}(\xi))$ and $\zeta$ goes sharply through it). Then $p_d$ is the last free point lying both on $\xi$ and $\gamma$, and the existence of a $q_d \in \mathcal{R}^{p_d}(\xi)$ satisfying $I_{\xi}(q_d) = \frac{[\gamma.\xi]}{e_O(\gamma)} = I_{\xi}(d)$ is guaranteed by Proposition \ref{pro-polar-inv}. Moreover, Proposition \ref{pro-polar-inv} also says that for any rupture point $q$ there exists a branch $\gamma '$ (not necessarily unique) of $\zeta$ such that $I_{\xi}(q) = \frac{[\gamma'.\xi]}{e_O(\gamma')}$, and that $q$ is satellite of the last free point lying both on $\xi$ and $\gamma'$. So it only remains to prove that the same branch $\gamma$ cannot work for several rupture points, which is equivalent to prove the uniqueness of $q_d$.

The case $p_d = O$ is quite easy, since $O$ has no $O$-satellite points, and thus $q_d = O$ is the only possibility.

For the rest of the proof assume $p_d \neq O$, and suppose  that $q_1 \prec q_2$ are two rupture points of $\xi$ equal to or satellite of $p_d$ and such that $I_{\xi}(q_1) = I_{\xi}(q_2) = I_{\xi}(d)$. By Proposition \ref{pro-growing}, no branch of $\xi$ can be bigger than $q_1$. But since $q_2$ is a rupture point, there exists a branch of $\xi$ going through $q_2$ and having a free point in its first neighbourhood, and such a branch is clearly bigger than $q_1$, which leads to a contradiction. Therefore, there exists a unique $p_d$-satellite rupture point $q_d$ satisfying $I_{\xi}(q_d) = I_{\xi}(d)$.

In order to prove that $q_d \preceq \gamma$, which is equivalent to $q_d \preceq d$, note that we can consider $\frac{[\gamma.\xi]}{e_O(\gamma)}$ as $I_{\xi}(q')$, where $q'$ is the last $p_d$-satellite point on $\gamma$ (because $p_d$ is the last free point lying both on $\gamma$ and $\xi$). Then $I_{\xi}(q_d) = I_{\xi}(q')$, and again by Proposition \ref{pro-growing} we obtain that $q_d \preceq q'$, which implies $q_d \preceq \gamma$ by definition.
\end{proof}

\begin{cor}
The number of rupture points of a reduced  singular curve $\xi$ is bounded above by the number of dicritical points of $BP(\mathcal{J}(\xi))$.
\end{cor}

From now on, if $d \in \mathcal{D}$ is a dicritical point of $BP(\mathcal{J}(\xi))$, $p_d$ will denote the last free point lying both on $\xi$ and $\mathcal{K}(d)$, and $q_d$ will stand for the rupture point associated to $d$ according to Proposition \ref{Prop-p-gamma}. Note that $q_d$ may be either equal to or satellite of $p_d$. As a particular case, if $O \in \mathcal{D}$, then $q_O = O$ because it is the only point $\preceq O$. However, determining $q_d$ in the case $d \neq O$, which we assume from now on, is not so easy and needs some more work.

The first step to determine $q_d$ is to compute the polar invariant $I_{\xi}(q_d)= I_{\xi}(d) = \frac{[\xi.\mathcal{K}(d)]}{\nu_O(\mathcal{K}(d))}$ from $BP(\mathcal{J}(\xi))$, and we can do it thanks to the following

\begin{lem} \label{Lem-Rec-Pol-Inv}
If $d \in \mathcal{D}$ is a dicritical point of $BP(\mathcal{J}(\xi))$,  then $I_{\xi}(d) = \frac{[BP(\mathcal{J}(\xi)).\mathcal{K}(d)]}{n_d} + 1$.
\end{lem}
\begin{proof}
Let $\gamma$ be a branch of a topologically generic transverse polar $\zeta$ of $\xi$ going sharply through $\mathcal{K}(d)$. So, proving the statement is equivalent to prove
$$I_{\xi}(q_d) = \frac{[\xi.\gamma]}{e_O(\gamma)} = \frac{[BP(\mathcal{J}(\xi)).\gamma]}{e_O(\gamma)} + 1.$$
By definition, there exists some equation $f$ of $\xi$ and some smooth germ $g = 0$ such that $\zeta$ is given by the equation $\frac{\partial(f,g)}{\partial(x,y)} = 0$. Up to change of coordinates, we may assume $g = x$, and thus $\zeta : \frac{\partial f}{\partial y} = 0$.

Since $BP(\mathcal{J}(\xi)) = BP\left(\frac{\partial f}{\partial x}, \frac{\partial f}{\partial y}\right)$, all but finitely many germs $\zeta'$ of the pencil
$$\left\{ \alpha \frac{\partial f}{\partial x} + \beta \frac{\partial f}{\partial y} = 0 \right\}$$
go sharply through $BP(\mathcal{J}(\xi))$ and miss the first point lying on $\gamma$ and not in $BP(\mathcal{J}(\xi))$. Then, for any such $\zeta'$, we have $[BP(\mathcal{J}(\xi)).\gamma] = [\zeta'.\gamma]$. Moreover, up to a linear change of the coordinate $y$, we may assume that $\zeta' : \frac{\partial f}{\partial x} = 0$.

Now, let $n = e_O(\gamma)$ and let $s(x)$ be a Puiseux series of $\gamma$. Thus, we have (see \cite[Remark 2.6.6]{Cas00} for this formula of the intersection product)
\[[\xi.\gamma] = \sum_{\epsilon^n = 1} o_x(f(x,s(\epsilon x))) \, \text{ and } \, [\zeta'.\gamma] = \sum_{\epsilon^n = 1} o_x\left(\frac{\partial f}{\partial x}(x,s(\epsilon x))\right).\]
We may relate the summands in the two formulas as follows:
\[o_x(f(x,s(\epsilon x))) = 1 + o_x\left(\frac{d}{dx}f(x,s(\epsilon x))\right) = 1 + o_x\left(\frac{\partial f}{\partial x}(x,s(\epsilon x)) + \epsilon \frac{\partial f}{\partial y}(x,s(\epsilon x)) s'(x) \right)\]
and since $\gamma$ is a branch of $\zeta : \frac{\partial f}{\partial y} = 0$, the summand $\epsilon \frac{\partial f}{\partial y}(x,s(\epsilon x)) s'(x)$  vanishes identically. Now adding-up all these equalities for every $n$-th root of the unity $\epsilon$, we finally obtain
\[[\xi.\gamma] = n + [\zeta'.\gamma] = e_O(\gamma) + [BP(\mathcal{J}(\xi)).\gamma]\]
and the claim follows.
\end{proof}

The second step in order to determine $q_d$ is to determine $p_d$, the last free point preceding or equal to $q_d$, or equivalently, the last free point lying both on $\xi$ and $\mathcal{K}(d)$. To achieve this we will use a property that relates $p_d$ to the polar invariant $I_{\xi}(d)$:

\begin{pro} \label{Prop-p}
Let $d \neq O$ be a dicritical point of $BP(\mathcal{J}(\xi))$, $q_d$ its associated rupture point (see Proposition \ref{Prop-p-gamma}), and $p_d \leq q_d$ the last free point preceding or equal to $q_d$. Let $p'_d<d$ be the last point such that $\frac{m_{p'_d}}{n_{p'_d}} < I_{\xi}(d)$ and whose next point in $\mathcal{K}(d)$ is free. Then $p_d$ is the next point of $p'_d$ in $\mathcal{K}(d)$. In particular, $p_d \in BP\left(\mathcal{J}\left(\xi\right)\right)$.
\end{pro}
\begin{proof}
Suppose $p_d$ is proximate to $p'$. We want to show that $p'=p'_d$ as defined in the statement. Since $q_d \preceq d$, we must have $p_d \leqslant d$, and hence $p' < d$. Moreover, combining Proposition \ref{pro-growing} and Theorem \ref{Thm-values-versus-alcades} we obtain that $v_{p'}(\xi) = m_{p'}$ and
\[I_{\xi}(p') = \frac{m_{p'}}{n_{p'}} < I_{\xi}(q_d).\]
So, among all points strictly preceding $d$ whose next point in $\mathcal{K}(d)$ is free, $p'$ must satisfy $\frac{m_{p'}}{n_{p'}} < I_{\xi}(d)$. We need to show that indeed $p'$ is the last point with such property. Let $O<p_1<p_2<\ldots<p_k$ be the free points in $\mathcal{K}(d)$, and for each $i \geq 1$ let $p_i'$ be the point immediatly preceding $p_i$. Then $p_{i+1}'$ is either equal to or satellite of $p_i$, and hence Proposition \ref{pro-growing} gives
\[I_{\xi}(p_i') \stackrel{(a)}{\leq} I_{\xi}(p_{i+1}') \leq I_{\xi}(p_i) \qquad \textrm{ for all } 1 \leq i < k,\]
where the inequality $(a)$ is strict if and only if  $p_i' \leq p'$, since this is equivalent to $p_i$ lying on $\xi$. In particular, the sequence $\{I_{\xi}(p_i')\}_{i=1}^k$ is strictly increasing up to $p'$, and it becomes constant after that.

Suppose now to get a contradiction that $p' = p_r'$ for some $1 \leq r \leq k$, but that it is not the last $p_i'$ such that $\frac{m_{p_i'}}{n_{p_i'}} < I_{\xi}(d)$, i.e. assume $r < k$ and $\frac{m_{p_s'}}{n_{p_s'}} < I_{\xi}(d)$ for some $r < s \leq k$. This implies that
\[I_{\xi}(p_s') \leq \frac{m_{p_s'}}{n_{p_s'}} < I_{\xi}(d),\]
but since $p_r$ is the last free point lying both on $\xi$ and $\mathcal{K}(d)$, it holds the equality $I_{\xi}(p_s') = I_{\xi}(d)$, which leads to a contradiction and we are done.
\end{proof}

Now that $p_d$ has been determined, it only remains to know which of its satellite points is $q_d$. The problem is that there might be many points $q$, equal to or satellite of $p_d$, with the same invariant quotient $I_{\xi}(q) =I_{\xi}(d)$. Moreover, although 
Proposition \ref{pro-growing} implies that $q_d$ is the smallest (by $\prec$) such point, there is no way to determine it explicitly from the last $p_d$-satellite point in $\mathcal{K}(d)$. Fortunately, the $p_d$-satellite points $q$ bigger than $q_d$ and with the same invariant are exactly the points for which $v_q(\xi) < m_q$ (Theorem \ref{Thm-values-versus-alcades}), and this fact enables us to solve this case. In other words, the heights $m_q$ can distinguish between the $p_d$-satellite points when the invariants $I_{\xi}(q)$ cannot. This fact allows us to develop an algorithm which computes $q_d$ just from the polar invariant $I_{\xi}(d)$ and the already determined $p_d$, by seeking the unique point $q$ which is either equal to or satellite of $p_d$ and for which the equality $\frac{m_q}{n_q} = I_{\xi}(q_d) = I_{\xi}(d)$ holds. In fact, it computes step by step all the intermediate points $p_d = q_0<q_1<\cdots<q_{k-1}<q_k = q_d$ (where $q_i$ is in the first neighbourhood of $q_{i-1}$).

The procedure works as follows:
\begin{itemize}
\item Start with $i = 0$ and $q_0 = p_d$.
\item While $\frac{m_{q_i}}{n_{q_i}} \neq I_{\xi}(d)$ do
    \begin{itemize}
    \item If $\frac{m_{q_i}}{n_{q_i}} > I_{\xi}(d)$ take $q_{i+1}$ to be the first satellite of $q_i$.
    \item If $\frac{m_{q_i}}{n_{q_i}} < I_{\xi}(d)$ take $q_{i+1}$ to be the second satellite of $q_i$.
    \end{itemize}
    Increase $i$ to $i+1$.
\item If $\frac{m_{q_i}}{n_{q_i}} = I_{\xi}(d)$, end by taking $k = i$ and $q_d = q_k$.
\end{itemize}

\begin{thm} \label{Thm-Calcul-q_gamma}
Keep the above notations. The above procedure ends after a finite number of steps, and actually computes the rupture point $q_d$.
\end{thm}
\begin{proof}
First of all, note that since $q_d$ is a rupture point, Corollary \ref{Cor-rupt-valor=alcada} implies that $I_{\xi}(d) = I_{\xi}(q_d) =\frac{m_{q_d}}{n_{q_d}}$. Therefore, since there are finitely many points between $p_d$ and $q_d$, it is enough to check that each $q_i$ actually precedes $q_d$ and that if $\frac{m_{q_i}}{n_{q_i}} = I_{\xi}(d)$ then $q_i = q_d$.

To see that $q_i \leqslant q_d$ for each $i$ we use induction on $i$.  For $i = 0$, we have $q_0 = p_d$, and hence $q_0 = p_d \leqslant q_d$ by definition of $p_d$. Now suppose we have reached the step $i$ of the algorithm and we have to perform another step. This means that $q_i \leqslant q_d$ and $\frac{m_{q_i}}{n_{q_i}} \neq I_{\xi}(d)$. We know that in this case $q_i < q_d$, and we claim that the point $q_{i+1}$ computed by the algorithm still precedes $q_d$. Indeed, since $p_d \leqslant q_i < q_d$ and $q_d$ is $p_d$-satellite, the point in the first neighbourhood of $q_i$ preceding $q_d$ must be satellite. Hence, it only remains to check that the choice made by the algorithm is the correct one.
\begin{itemize}
\item If $\frac{m_{q_i}}{n_{q_i}} < I_{\xi}(d)$, then $I_{\xi}(q_i) = \frac{v_{q_i}(\xi)}{n_{q_i}} \leqslant \frac{m_{q_i}}{n_{q_i}} < I_{\xi}(d)$ by Theorem \ref{Thm-values-versus-alcades}. Therefore, by Lemma \ref{lem-Ordre-p-satellites} and Proposition \ref{pro-growing}, the next point $q_{i+1}$ must be the second satellite point, for if it was the first one the invariants $I_{\xi}(q)$ would be strictly smaller than $I_{\xi}(d)$ for every satellite $q \geqslant q_{i+1}$.
\item If $\frac{m_{q_i}}{n_{q_i}} > I_{\xi}(d)$, then either $I_{\xi}(d) < I_{\xi}(q_i) \leqslant \frac{m_{q_i}}{n_{q_i}}$ or $I_{\xi}(q_i) \leqslant I_{\xi}(d) < \frac{m_{q_i}}{n_{q_i}}$. In the former case we apply Lemma \ref{lem-Ordre-p-satellites} and Proposition \ref{pro-growing} as above to see that $q_{i+1}$ must be the first satellite point of $q_i$. In the latter case we have that $v_{q_i}(\xi) < m_{q_i}$, and hence by Theorem \ref{Thm-values-versus-alcades} every branch of $\xi$ through $p_d$ is smaller than $q_i$. This implies in particular that $q_d \prec q_i$, and thus by Lemma \ref{lem-Ordre-p-satellites} $q_d$ must be infinitely near to the first satellite of $q_i$.
\end{itemize}
In any case, the algorithm is correct.

In order to complete the proof, we must check that  the algorithm does not stop before reaching the point $q_d$. That is, we have to show that if $q$ is either $p_d$ or any $p_d$-satellite point strictly preceding $q_d$, then $\frac{m_q}{n_q} \neq I_{\xi}(q_d)$.
\begin{itemize}
\item If $q \prec q_d$, any branch of $\xi$ going through $q_d$ is bigger than $q$. Then Proposition \ref{pro-growing} implies that $I_{\xi}(q) < I_{\xi}(d)$, and by Theorem \ref{Thm-values-versus-alcades} we also have that $v_q(\xi) = m_q$. So $I_{\xi}(q) = \frac{m_q}{n_q} < I_{\xi}(d)$ and in particular $\frac{m_q}{n_q} \neq I_{\xi}(d)$.
\item Consider now the case $q \succ q_d$. Then, on the one hand Proposition \ref{pro-growing} implies that  $I_{\xi}(d) \leqslant I_{\xi}(q)$, with equality if and only if every branch of $\xi$ going through $p_{d}$ is not bigger than $q_d$. On the other hand, Theorem \ref{Thm-values-versus-alcades} says that $v_q(\xi) \leqslant m_q$, and equality holds if and only if there is some branch of $\xi$ not smaller than $q$. Summarizing, we have $I_{\xi}(d) \leqslant I_{\xi}(q) \leqslant \frac{m_q}{n_q}$, and having equality $I_{\xi}(d) = \frac{m_q}{n_q}$ would imply (by Theorem \ref{Thm-values-versus-alcades}) that there is some branch of $\xi$ through $p_d$ which is not smaller than $q$. But such a branch would be bigger than $q_d$, implying (by Proposition \ref{pro-growing}) that $I_{\xi}(d) < I_{\xi}(q) \leqslant \frac{m_q}{n_q}$ and thus contradicting the equality $I_{\xi}(d) = \frac{m_q}{n_q}$.
\end{itemize}
\end{proof}

\subsection{Recovering values} \label{ssect_rec_val}

This section is devoted to explain how the values of a curve $\xi$ at its singular points can be recovered from the invariants $m_p$ and $n_p$, provided the set of rupture points $\mathcal{R}(\xi)$ (and hence the set of singular points $\mathcal{S}(\xi)$) is already known.
Recall that from Lemma \ref{alternativa-I(p)} we already know that $v_p(\xi) = n_p I_{\xi}(p)$ at any $p \in \mathcal{S}(\xi)$, but that the difficulty lies on the computation of the invariant quotient $I_{\xi}(p)$.


Assume first that $p \in \mathcal{R}\left(\xi\right)$ is a rupture point. Then Corollary \ref{Cor-rupt-valor=alcada} implies that $v_p(\xi) = m_p$.

Suppose now that $p \in \mathcal{S}(\xi)$ a free singular point which is not a rupture point. By Theorem \ref{Thm-values-versus-alcades}, we have the equality $v_p(\xi) = m_p$ if and only if there is a free point in the first neighbourhood of $p$ lying on $\xi$. In particular, if there is a free singular point in the first neighbourhood of $p$, we can also assert that $v_p(\xi) = m_p$. If otherwise there is no free singular point on $\xi$ in the first neighbourhood of $p$, then there is at most one free point lying on $\xi$ in the first neighbourhood of $p$ and, if it exists, it is non-singular. If there is no such a point, then 
Proposition \ref{pro-growing} implies that
\[v_p(\xi) = n_pI_{\xi}(p) = n_pI_{\xi}(q) = \frac{n_p}{n_q}v_q(\xi) = \frac{n_p}{n_q}m_q,\]
where $q$ is the biggest $p$-satellite point in $\mathcal{R}(\xi)$. On the contrary, if $\xi$ has a free point in the first neighbourhood of $p$, then Proposition \ref{pro-I-satellite-3}, Lemma \ref{alternativa-I(p)} and Corollary \ref{Cor-rupt-valor=alcada} give the inequalities
\[\frac{v_p(\xi)-1}{n_p} < \frac{v_q(\xi)}{n_q} = \frac{m_q}{n_q} < \frac{v_p(\xi)}{n_p},\]
which are equivalent to
\[\frac{n_p}{n_q}m_q < v_p(\xi) < \frac{n_p}{n_q}m_q+1,\]
where as before $q$ is the biggest $p$-satellite point in $\mathcal{R}(\xi)$. Hence, in any case, $v_p(\xi)$ belongs to the real interval $\left[\frac{n_p}{n_q}m_q,\frac{n_p}{n_q}m_q+1\right)$. Since the width of this interval is one, there is exactly one integer in it, and thus the value $v_p(\xi)$ is uniquely determined.

So far we have proved the following

\begin{pro} \label{Pro-value-free} Let $p \in \mathcal{S}(\xi)$ be a free singular point which is not a rupture point.
\begin{itemize}
\item If there is a free singular point in the first neighbourhood of $p$, then $v_p(\xi) = m_p$.
\item Otherwise, let $q$ be the biggest point in $\mathcal{R}^p(\xi)$ (which must be non-empty). Then $v_p(\xi)$ is the only integer in the interval \[\left[\frac{n_p}{n_q}m_q,\frac{n_p}{n_q}m_q+1\right).\] Moreover, the equality $v_p(\xi) = \frac{n_p}{n_q}m_q$ holds if and only if there is no branch of $\xi$ going through $p$ and whose point in the first neighbourhood of $p$ is free.
\end{itemize}
\end{pro}

It only remains to consider the case of satellite points $p \in \mathcal{S}(\xi)$ which are not rupture points, and it is solved by the next

\begin{pro} \label{Pro-value-satellite}
Let $p \in \mathcal{S}(\xi)$ be a satellite point of $\xi$ which is not a rupture point. Suppose moreover that $p$ is satellite of $p' \in \mathcal{S}(\xi)$ and let $q$ be the biggest point in $\mathcal{R}^{p'}(\xi)$. Then
\[v_p(\xi) =
\begin{cases}
\frac{n_p}{n_{p'}}v_{p'}(\xi) & \mbox{ if $p \succ q$ and $v_{p'}(\xi) = \frac{n_{p'}}{n_q}m_q$,} \\
m_p & \mbox{ otherwise.} \\
\end{cases}\]
\end{pro}
\begin{proof}
If $p' = q$ is a rupture point, there exists a branch of $\xi$ going through $p'$ and having a free point in its first neighbourhood, and the same holds if otherwise $p' \neq q$ but $v_{p'}(\xi) \neq \frac{n_{p'}}{n_q}m_q$ (by Proposition \ref{Pro-value-free}). Thus, in any case Theorem \ref{Thm-values-versus-alcades} implies that $v_p(\xi) = m_p$.

Suppose now that $p'$ is not a rupture point and $v_{p'}(\xi) = \frac{n_{p'}}{n_q}m_q$. Then there is no branch of $\xi$ going through $p'$ and having a free point in its first neighbourhood. If furthermore $p \prec q$, Theorem \ref{Thm-values-versus-alcades} applies to give $v_p(\xi) = m_p$ again, but if otherwise $p \succ q$, 
Proposition \ref{pro-growing} gives that
\[v_p(\xi) = n_pI_{\xi}(p) = n_pI_{\xi}(p') = \frac{n_p}{n_{p'}}v_{p'}(\xi).\]
\end{proof}

As a consequence of the proof of Proposition \ref{Pro-value-satellite} we infer the following result, which determines those free points $p \in \mathcal{S}(\xi)$ (besides the rupture points) admitting branches of $\xi$ going through $p$ and non-singular after $p$.


\begin{cor}
Let $p \in \mathcal{S}(\xi)$ be a free singular point. Then there is some branch of $\xi$ non-smaller than $p$ if and only if either $p$ is a rupture point or $v_p(\xi) \neq \frac{n_p}{n_q} m_q$ (where $q$ is the biggest $p$-satellite rupture point of $\xi$).
\end{cor}

\subsection{The algorithm}

\label{sect_alg}




\begin{alg} \label{alg-1}
Starting from the weighted cluster $BP\left(\mathcal{J}\left(\xi\right)\right)$, the following algorithm computes the sets $\mathcal{R} = \mathcal{R}(\xi)$ and $\mathcal{S} = \mathcal{S}(\xi)$ of rupture and singular points of $\xi$, together with the values $v_p = v_p(\xi)$ for any $p \in \mathcal{S}(\xi)$.

\bigskip

\noindent Part 1: {\em Recovering the rupture points and the singular points.}
\begin{enumerate}
    \item Start with $\mathcal{R} = \mathcal{S} = \emptyset$, and let $\mathcal{D}$ be the set of dicritical points of $BP(\mathcal{J}(\xi))$.
    \item If $O \in \mathcal{D}$, then set $\mathcal{R} = \mathcal{S} = \{O\}$.
    \item For each $d \in \mathcal{D} - \{O\}$:
    \begin{enumerate}
        \item Compute $I = \frac{[BP(\mathcal{J}(\xi)).\mathcal{K}(d)]}{n_{d}} + 1$.
        \item Find the last point $p' < d$ such that $\frac{m_{p'}}{n_{p'}} < I$ and its next point $p$ in $\mathcal{K}(d)$ is free.
        \item Take $i = 0$ and $q_0 = p$.
        \item While $\frac{m_{q_i}}{n_{q_i}} \neq I$ do
        \begin{itemize}
            \item If $\frac{m_{q_i}}{n_{q_i}} > I$, take $q_{i+1}$ to be the first  satellite of $q_i$.
            \item If $\frac{m_{q_i}}{n_{q_i}} < I$, take $q_{i+1}$ to be the second satellite of $q_i$.
        \end{itemize}
        Increase $i$ to $i+1$.
        \item If $\frac{m_{q_i}}{n_{q_i}} = I$, set $\mathcal{R} = \mathcal{R} \cup \{q_i\}$ and $\mathcal{S} = \mathcal{S} \cup \{q \, | \, q \leqslant q_i\}$.
    \end{enumerate}
\end{enumerate}

\noindent Part 2: {\em Recovering the values.}
    \begin{enumerate}
    \item For each $p \in \mathcal{R}$ set $v_p = m_p$.
    \item For each free point $p \in \mathcal{S}-\mathcal{R}$
        \begin{itemize}
        \item If there is a free point both in $\mathcal{S}$ and in the first neighbourhood of $p$, set $v_p = m_p$.
        \item Otherwise, let $q$ be the biggest $p$-satellite point in $\mathcal{R}$ and set $v_p$ the only integer in the interval $\left[\frac{n_p}{n_q}m_q,\frac{n_p}{n_q}m_q+1\right)$.
        \end{itemize}
    \item For each satellite point $p \in \mathcal{S}-\mathcal{R}$, let $p'$ be the free point of which $p$ is satellite, and let $q$ be the biggest point in $\mathcal{R}$ which is either equal to or satellite of $p'$.

        \begin{itemize}
        \item If $p \succ q$ and $v_{p'} = \frac{n_{p'}}{n_q}m_q$ both hold, set $v_p = \frac{n_p}{n_{p'}}v_{p'}$.
        \item Otherwise, set $v_p = m_p$.
        \end{itemize}
    \end{enumerate}
\end{alg}



\begin{rmk} \label{cor-demo-altern}
This algorithm gives a proof of the first statement in Theorem \ref{Casas-8.6.4}. Furthermore, it is obvious that the algorithm yields similar clusters if it is applied to similar clusters, so in fact it also proves the second statement in Theorem \ref{Casas-8.6.4}, as we wanted.
\end{rmk}


\begin{cor} \label{Cor-4.11}
The cluster of singular points $\mathcal{S}(\xi)$ of any reduced singular curve $\xi : f=0$ is determined and may be explicitly computed from any two polars $P_{g_1}(f)$ and $P_{g_2}(f)$, provided $g_1$ and $g_2$ have different tangents, regardless whether they are topologically generic or even transverse ones.
\end{cor}
\begin{proof}
Note that $BP(\mathcal{J}(\xi)) = BP\left(\frac{\partial f}{\partial x},\frac{\partial f}{\partial y}\right) = BP\left(P_{g_1}(f),P_{g_2}(f)\right)$ for any two polars along different directions. This weighted cluster can be explicitly computed using the algorithm in \cite{Alb2004-CommAlg} valid for any pencil of curves. Then use Algorithm \ref{alg-1}.
\end{proof}

In some cases, the rupture point $q_d$ can be directly characterized from $p_d$ as the following Proposition shows.


\begin{pro} \label{Pro-q_gamma-facil}
Let $d \in \mathcal{D}$ be a dicritical point of $BP(\mathcal{J}(\xi))$ with polar invariant $I= I_{\xi}(d)$, and suppose $p_d$ is the last free point lying both on $\xi$ and $\mathcal{K}(d)$. Assume that there exists another dicritical point $d' \in \mathcal{D}$ for which $p_{d'} = p_d$ but whose polar invariant $I'= I_{\xi}(d')$ is greater than $I$. Then $q_d$ is the last $p_d$-satellite point in $\mathcal{K}(d)$.
\end{pro}

\begin{proof}
Suppose the claim is false and let $\bar{q}_d$ be the last $p_d$-satellite point in $\mathcal{K}(d)$. Proposition \ref{Prop-p-gamma} implies that $\bar{q}_d \succeq q_d$, and hence $\bar{q}_d \succ q_d$. Moreover, since $p_d$ is the last free point lying both on $\xi$ and $\mathcal{K}(d)$, we can take indistinctly  $\gamma^d$ or $\gamma^{\bar{q}_d}$ to compute
\[I_{\xi}(\bar{q}_d) = \frac{[\xi.\gamma^{\bar{q}_d}]}{e_O(\gamma^{\bar{q}_d})} = \frac{[\xi.\gamma^d]}{e_O(\gamma^d)} = I.\]

If $q_{d'}$ is the rupture point associated to $d'$, we claim that $q_{d'} \succ q_d$. Indeed, if it is not the case, Proposition \ref{pro-growing} would imply that $I' = I_{\xi}(q_{d'}) \leqslant I_{\xi}(\bar{q}_d) = I$ contradicting our hypothesis. Therefore, there exists some branch of $\xi$ bigger than $q_d$, and then Proposition \ref{pro-growing} again will give $I_{\xi}(q_d) < I_{\xi}(\bar{q}_d) = I$, which contradicts that $q_d$ is the rupture point associated to $d$.
\end{proof}

Based on Proposition \ref{Pro-q_gamma-facil}, we present an alternative version of the algorithm for the part of recovering the rupture and the singular points. This apparently longer version gives a more precise and geometrical description of some of the rupture points $q_d$, for which also avoids the tedious task of performing the iterations in step (d).

\begin{alg} \label{alg-2} Part 1 of Algorithm \ref{alg-1} may be replaced by the following:
\begin{enumerate}
    \item Start with $\mathcal{R} = \mathcal{S} = \emptyset$, and let $\mathcal{D}$ be the set of dicritical points of $BP(\mathcal{J}(\xi))$.
    \item If $O \in \mathcal{D}$, then set $\mathcal{R} = \mathcal{S} = \{O\}$
    \item For each $d \in \mathcal{D} - \{O\}$ compute $I_d = \frac{[BP(\mathcal{J}(\xi)).\mathcal{K}(d)]}{n_d}+1$, and order $\mathcal{D}-\{O\} = \{d_1,\ldots,d_k\}$ by descending order of $I_d$ (i.e., $I_{d_1} \geqslant \ldots \geqslant I_{d_k})$.
    \item For each $j = 1, \ldots , k$ do:
    \begin{enumerate}
        \item Find the last point $p' < d_j$ such that $\frac{m_{p'}}{n_{p'}} < I_{d_j}$ and its next point $p$ in $\mathcal{K}(d_j)$ is free.
        \item If $p$ has already appeared at this step, let $q_j$ be the last $p$-satellite point in $\mathcal{K}(d_j)$ and set $\mathcal{R} = \mathcal{R} \cup \{q_j\}$ and $\mathcal{S} = \mathcal{S} \cup \{q \, | \, q \leqslant q_j\}$. Then skip to the next $j$.
        \item Otherwise, take $i = 0$ and $q_0 = p$.
        \item While $\frac{m_{q_i}}{n_{q_i}} \neq I_{d_j}$ do
            \begin{itemize}
            \item If $\frac{m_{q_i}}{n_{q_i}} > I_{d_j}$, take $q_{i+1}$ to be the first satellite of $q_i$.
            \item If $\frac{m_{q_i}}{n_{q_i}} < I_{d_j}$, take $q_{i+1}$ to be the second satellite of $q_i$.
            \end{itemize}
            Increase $i$ to $i+1$.
        \item If $\frac{m_{q_i}}{n_{q_i}} = I_{d_j}$, set $\mathcal{R} = \mathcal{R} \cup \{q_i\}$ and $\mathcal{S} = \mathcal{S} \cup \{q \, | \, q \leqslant q_i\}$.
        \end{enumerate}
    \end{enumerate}
\end{alg}

\subsection{Examples} \label{sect_examp}

Let us illustrate through some examples the application of Algorithm \ref{alg-1}.
We work each example as follows: we start from an equation $f$ of $\xi$ and then we present our initial data, the weighted cluster of base points \mbox{$BP(\mathcal{J}(\xi)) = BP\left(\frac{\partial f}{\partial x},\frac{\partial f}{\partial y}\right)$}, which has been computed using the algorithm given in \cite{Alb2004-CommAlg} (this part will not be explained in any case). Then we apply Algorithm \ref{alg-1} to $BP(\mathcal{J}(\xi))$ in order to recover the cluster $\mathcal{S}(\xi)$ with the corresponding values, showing the invariants $\frac{m_p}{n_p}$ computed and explaining how the algorithm works. At the end, it can be checked that our output coincides with $\mathcal{S}(\xi)$.

For each example of singular curve $\xi $, four Enriques diagrams will be shown: the first one shows the equisingularity class of the original curve $\xi$. The second one contains the names of the singular points of $\xi$ and the base points of $\mathcal{J}(\xi)$, where the dots in each square mean that there are as many free points as the number in the same square. The third diagram represents the cluster $BP(\mathcal{J}(\xi))$ with its virtual multiplicities, and the fourth one shows the heights of the trunks $m_p$ and the multiplicities $n_p$ of the morphism $\varphi _p$ for each $p \in \mathcal{S}(\xi) \cup BP(\mathcal{J}(\xi))$ (which are computed using Lemma \ref{calcul-n_p-m_p}). The points lying on $\xi$ are represented with black filled circles, while the circles representing points not lying on $\xi$ are filled in white. When reading each example, it is advisable to look at the corresponding figure in order to fix some notation, paying attention to the labels of the points of the clusters.

We start with a pair of simple examples, which are classical in the  literature about polars and were given by Pham \cite{Pha71} in order to prove that the equisingularity class of a curve does not determine the equisingularity class of its topologically generic polars. Namely, the curve $\xi $ of Example \ref{Ex-04} and that of Example \ref{Ex-05} are equisingular, while its topologically generic polars are not. Observe that nor are similar their respective clusters $BP(\mathcal{J}(\xi))$, proving also that the reciprocal of Theorem \ref{Casas-8.6.4} does not hold.

\begin{exam} [See Figure \ref{fig-ex-04}] \label{Ex-04}
Take $\xi$ to be given by $y^3-x^{11} +\alpha x^8y = 0$, with $\alpha \neq 0$. It is irreducible and has only one characteristic exponent: $\frac{11}{3}$. The cluster $BP(\mathcal{J}(\xi))$ is shown in Figure \ref{fig-ex-04}, and hence topologically generic polars of $\xi$ consist of two smooth branches sharing the points on $\xi$ up to $p_3$, the point on $\xi$ in the third neighbourhood of $O$. Moreover, topologically generic polars of $\xi$ share four further fixed free points after $p_3$, two on each branch.

\begin{figure}
\begin{center}
\psfrag{A}{$\xi, (e_p(\xi), v_p(\xi))$} \psfrag{A0}{\footnotesize $(3,3)$} \psfrag{A1}{\footnotesize $(3,6)$} \psfrag{A2}{\footnotesize $(3,9)$} \psfrag{A3}{\footnotesize $(2,11)$} \psfrag{A4}{\footnotesize $(1,21)$} \psfrag{A5}{\footnotesize $(1,33)$}

\psfrag{B}{$\mathcal{S}(\xi) \cup BP(\mathcal{J}(\xi))$} \psfrag{B0}{\footnotesize $O$} \psfrag{B1}{\footnotesize $p_1$} \psfrag{B2}{\footnotesize $p_2$} \psfrag{B3}{\footnotesize $p_3$} \psfrag{B4}{\footnotesize $p_4$} \psfrag{B5}{\footnotesize $p_5$} \psfrag{B6}{\footnotesize $p_6$} \psfrag{B7}{\footnotesize $p_7$} \psfrag{B8}{\footnotesize $p_8$} \psfrag{B9}{\footnotesize $p_9$}

\psfrag{C}{$BP(\mathcal{J}(\xi)), \nu_p$} \psfrag{C0}{\footnotesize $2$} \psfrag{C1}{\footnotesize $2$} \psfrag{C2}{\footnotesize $2$} \psfrag{C3}{\footnotesize $2$} \psfrag{C6}{\footnotesize $1$} \psfrag{C7}{\footnotesize $1$} \psfrag{C8}{\footnotesize $1$} \psfrag{C9}{\footnotesize $1$}

\psfrag{D}{$\mathcal{S}(\xi) \cup BP(\mathcal{J}(\xi)), \frac{m_p}{n_p}$} \psfrag{D0}{\footnotesize $\frac{3}{1}$} \psfrag{D1}{\footnotesize $\frac{6}{1}$} \psfrag{D2}{\footnotesize $\frac{9}{1}$} \psfrag{D3}{\footnotesize $\frac{12}{1}$} \psfrag{D4}{\footnotesize $\frac{21}{2}$} \psfrag{D5}{\footnotesize $\frac{33}{3}$} \psfrag{D6}{\footnotesize $\frac{14}{1}$} \psfrag{D7}{\footnotesize $\frac{14}{1}$} \psfrag{D8}{\footnotesize $\frac{16}{1}$} \psfrag{D9}{\footnotesize $\frac{16}{1}$}
\includegraphics[scale=0.15]{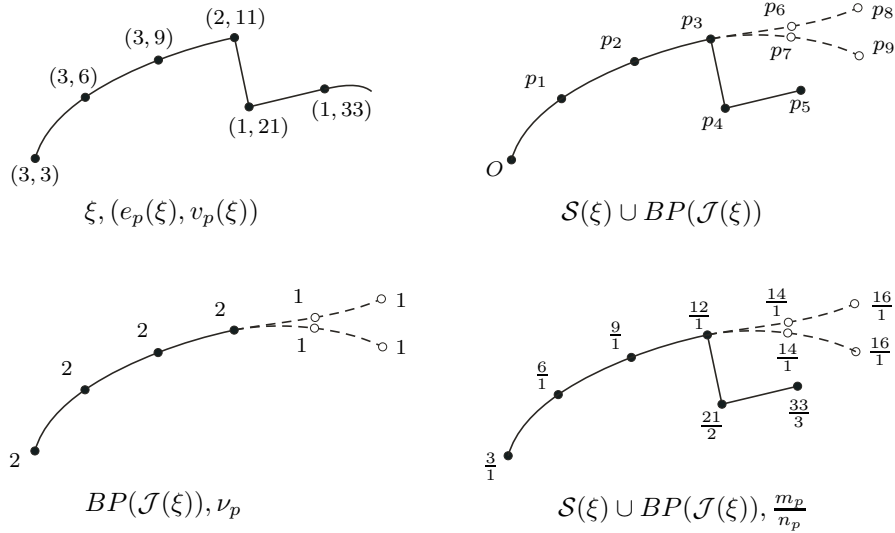}
\end{center}
\caption{\label{fig-ex-04} Enriques diagrams for the singular curve $\xi : y^3-x^{11}+\alpha x^8y = 0$ ($\alpha \neq 0$).}
\end{figure}

Since $O \not \in \mathcal{D} = \{p_8,p_9\}$, we start with $\mathcal{R} = \mathcal{S} = \emptyset$.  The polar invariants are
\[I = I_{p_8} = I_{p_9} = \frac{[BP(\mathcal{J}(\xi)).\mathcal{K}(p_8)]}{n_{p_8}}+1 = \frac{2\cdot1+2\cdot1+2\cdot1+2\cdot1+1^2+1^2}{1}+1 = 11.\]
We start with $p_8$. The corresponding point $p'$ is $p_2$, and thus the rupture point associated to $p_8$ is satellite of $q_0 = p_3$. Step 4(d) consists of the next two iterations:
\begin{itemize}
\item $\frac{m_{q_0}}{n_{q_0}} = 12 > 11 = I$, so we take $q_1 = p_4$, the first satellite of $p_3$.
\item $\frac{m_{q_1}}{n_{q_1}} = \frac{21}{2} < 11 = I$, so we take $q_2 = p_5$, the second satellite of $p_4$.
\end{itemize}
Since $\frac{m_{q_2}}{n_{q_2}} = 11 = I$, we end by taking $\mathcal{R} = \{p_5\}$ and $\mathcal{S} = \{O,p_1,\ldots,p_5\}$.

Taking $p_9$ we have $I_{p_9} = I = 11$ and again $p' = p_2$. Hence we obtain the same  results as for $p_8$ and it is not necessary to add any further point to $\mathcal{R}$ or $\mathcal{S}$.

The second part of the algorithm starts setting  $v_{p_5} = m_{p_5} = 33$. On the one hand, since there are free singular points in the first neighbourhood of $O, p_1$ and $p_2$, Step 2 yields $v_O = 3, v_{p_1} = 6$ and $v_{p_2} = 9$. On the other hand, since there are no free singular points in the first neighbourhood of $p_3$, the second instance of step 2 gives $v_{p_3} = 11$, the only integer in the interval
\[\left[\frac{n_{p_3}}{n_{p_5}}m_{p_5},\frac{n_{p_3}}{n_{p_5}}m_{p_5}+1\right) = [11,12).\]
Finally, the third step of the second part applies to recover $v_{p_4}$. Here $p'$ is $p_3$ and $q$ is $p_5$. Since $p_4 \prec p_5$, we must follow the second instance of step 3 and set $v_{p_4} = m_{p_4} = 21$.
\end{exam}

\begin{exam} [See Figure \ref{fig-ex-05}] \label{Ex-05}
Now consider the curve $\xi$ given by $y^3-x^{11} = 0$. It is again irreducible with single characteristic exponent $\frac{11}{3}$, and hence it is equisingular to the curve in the previous example (in fact, it corresponds to take $\alpha = 0$ in the equation of Example \ref{Ex-04}). However, the Enriques diagram of $BP(\mathcal{J}(\xi))$ is not equal to that in Example \ref{Ex-04}. In this case, topologically generic polars also consist of two smooth branches, but they share five points (instead of four, as happened in the previous example) and there are no more base points.

\begin{figure}
\begin{center}
\psfrag{A}{$\xi, (e_p(\xi), v_p(\xi))$} \psfrag{A0}{\footnotesize $(3,3)$} \psfrag{A1}{\footnotesize $(3,6)$} \psfrag{A2}{\footnotesize $(3,9)$} \psfrag{A3}{\footnotesize $(2,11)$} \psfrag{A4}{\footnotesize $(1,21)$} \psfrag{A5}{\footnotesize $(1,33)$}

\psfrag{B}{$\mathcal{S}(\xi) \cup BP(\mathcal{J}(\xi))$} \psfrag{B0}{\footnotesize $O$} \psfrag{B1}{\footnotesize $p_1$} \psfrag{B2}{\footnotesize $p_2$} \psfrag{B3}{\footnotesize $p_3$} \psfrag{B4}{\footnotesize $p_4$} \psfrag{B5}{\footnotesize $p_5$} \psfrag{B6}{\footnotesize $p_6$}

\psfrag{C}{$BP(\mathcal{J}(\xi)), \nu_p$} \psfrag{C0}{\footnotesize $2$} \psfrag{C1}{\footnotesize $2$} \psfrag{C2}{\footnotesize $2$} \psfrag{C3}{\footnotesize $2$} \psfrag{C6}{\footnotesize $2$}

\psfrag{D}{$\mathcal{S}(\xi) \cup BP(\mathcal{J}(\xi)), \frac{m_p}{n_p}$} \psfrag{D0}{\footnotesize $\frac{3}{1}$} \psfrag{D1}{\footnotesize $\frac{6}{1}$} \psfrag{D2}{\footnotesize $\frac{9}{1}$} \psfrag{D3}{\footnotesize $\frac{12}{1}$} \psfrag{D4}{\footnotesize $\frac{21}{2}$} \psfrag{D5}{\footnotesize $\frac{33}{3}$} \psfrag{D6}{\footnotesize $\frac{15}{1}$}
\includegraphics[scale=0.15]{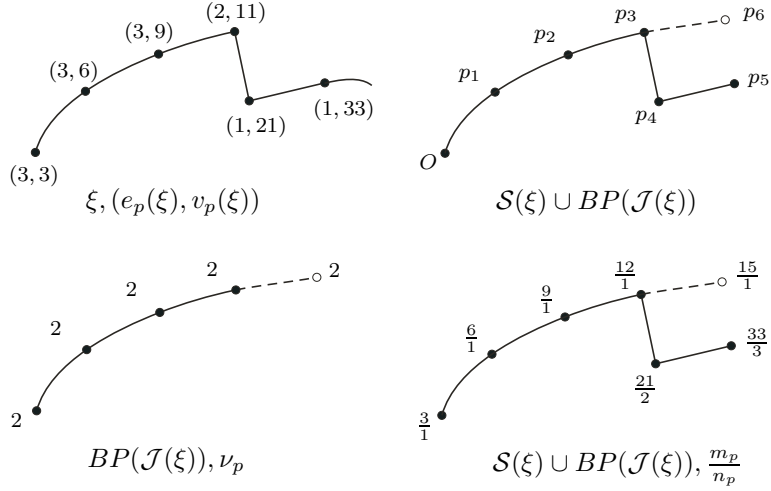}
\end{center}
\caption{\label{fig-ex-05} Enriques diagrams for the singular curve $\xi : y^3-x^{11} = 0$.}
\end{figure}

In this case there is only one dicritical point in $BP(\mathcal{J}(\xi))$: $p_6$, and its corresponding polar invariant is again
\[I = I_{p_6} = \frac{[BP(\mathcal{J}(\xi)).\mathcal{K}(p_6)]}{n_{p_6}}+1 = \frac{2\cdot1+2\cdot1+2\cdot1+2\cdot1+2\cdot1}{1}+1 = 11.\]
Moreover, the point $p$ is again $p_3$, and hence the algorithm works as it does in example 4 (recovering both the rupture points and the values).
\end{exam}

We expose now a more complicated example, since two of the branches of the curve have two characteristic exponents. After this example it will be clear that the computation of $\mathcal{S}(\xi)$ by hand is much faster using Algorithm \ref{alg-2}.

\begin{exam} (See figures \ref{fig-ex-C-A}, \ref{fig-ex-C-B}, \ref{fig-ex-C-C} and \ref{fig-ex-C-D})

\label{Ex-C}

Let $\xi$ be the curve with branches $\gamma_1, \ldots, \gamma_5$ given by the Puiseux series $s_1(x) = x^{\frac{11}{4}} + x^{\frac{51}{16}}, s_2(x) = x^{\frac{11}{4}} + x^{\frac{63}{20}}, s_3(x) = x^{\frac{8}{3}}, s_4(x) = x^{\frac{16}{7}}$ and $s_5(x) = x^{\frac{9}{4}}$. One possible equation for $\xi$ is
\begin{align*}
f = & (y^3 - x^8) (y^4 - x^9) (y^7 - x^{16}) \\
& (y^{16} - 4 x^{11} y^{12} - 80 x^{21} y^9 + 6 x^{22} y^8 - 72 x^{31} y^6 - \\
& \qquad 160 x^{32} y^5 - 4 x^{33} y^4 - 16 x^{41} y^3 + 56 x^{42} y^2 - 16 x^{43} y + x^{44} - x^{51}) \\
& (y^{20} - 5 x^{11} y^{16} + 10 x^{22} y^{12} - 140 x^{24} y^{12} - 10 x^{33} y^8 - 620 x^{35} y^8 - 110 x^{37} y^8 + \\
& \qquad 5 x^{44} y^4 - 260 x^{46} y^4 + 340 x^{48} y^4 - 20 x^{50} y^4 - x^{55} - 4 x^{57} - 6 x^{59} - 4 x^{61} - x^{63}).
\end{align*}
and its Enriques' diagram is shown in Figure \ref{fig-ex-C-A}. It is immediate that the set of rupture points of $\xi$ is $\mathcal{R}(\xi) = \{ p_4, p_5, p_7, p_8, p_{13}, p_{14} \}$.

\begin{figure}
\begin{center}
\psfrag{A}{$21$} \psfrag{B}{$43$} \psfrag{C}{$5$} \psfrag{D}{$9$} \psfrag{E}{$5$}

\psfrag{00}{\footnotesize$O$} \psfrag{01}{\footnotesize$p_1$} \psfrag{02}{\footnotesize$p_2$} \psfrag{03}{\footnotesize$p_3$} \psfrag{04}{\footnotesize$p_4$} \psfrag{05}{\footnotesize$p_5$} \psfrag{06}{\footnotesize$p_6$} \psfrag{07}{\footnotesize$p_7$} \psfrag{08}{\footnotesize$p_8$} \psfrag{09}{\footnotesize$p_9$} \psfrag{10}{\footnotesize$p_{10}$} \psfrag{\footnotesize12}{$p_{12}$}  \psfrag{13}{\footnotesize$p_{13}$} \psfrag{14}{\footnotesize$p_{14}$} \psfrag{15}{\footnotesize$p_{15}$} \psfrag{16}{\footnotesize$p_{16}$} \psfrag{17}{\footnotesize$p_{17}$} \psfrag{18}{\footnotesize$p_{18}$} \psfrag{19}{\footnotesize$p_{19}$} \psfrag{20}{\footnotesize$p_{20}$} \psfrag{21}{\footnotesize$p_{21}$}  \psfrag{22}{\footnotesize$p_{22}$} \psfrag{23}{\footnotesize$p_{23}$} \psfrag{24}{\footnotesize$p_{24}$} \psfrag{25}{\footnotesize$p_{25}$} \psfrag{26}{\footnotesize$p_{26}$} \psfrag{27}{\footnotesize$p_{27}$} \psfrag{28}{\footnotesize$p_{28}$} \psfrag{29}{\footnotesize$p_{29}$} \psfrag{30}{\footnotesize$p_{30}$} \psfrag{31}{\footnotesize$p_{11}$}
\includegraphics[scale=0.22]{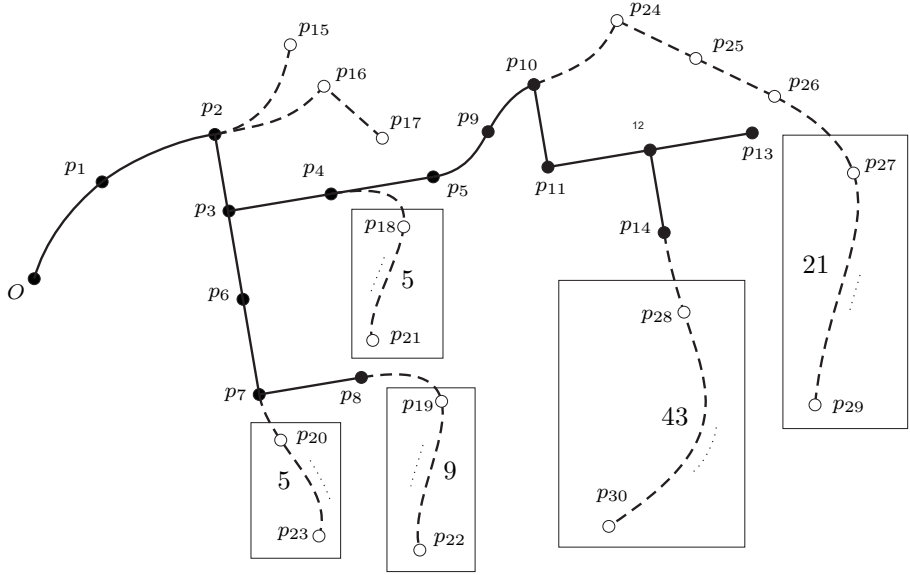}
\end{center}
\caption{\label{fig-ex-C-B} Singular points of $\xi$ and base points of $\mathcal{J}(\xi)$.}
\end{figure}

The representation of $BP(\mathcal{J}(\xi))$ in Figure \ref{fig-ex-C-C} shows in particular that topologically generic polars of $\xi$ have seven branches. One of the branches is smooth, four of them have only one characteristic exponent, and the two remaining branches have two characteristic exponents. This example also shows that $BP(\mathcal{J}(\xi))$ may contain a lot of points which are simple on the topologically generic polars.

Now we run the algorithm. Step 1 sets $\mathcal{R} = \mathcal{S} = \emptyset$, $\mathcal{D} = \{ p_{15}, p_{17}, p_{21}, p_{22}, p_{23}, p_{29}, p_{30} \}$, and since $O \not \in \mathcal{D}$ we go to step 3.

The polar invariants are $I_{15} = I_{17} = 132$, $I_{21} = 129$, $I_{22} = \frac{799}{7}$, $I_{23} = \frac{225}{2}$, $I_{29} = \frac{543}{4}$ and $I_{30} = \frac{678}{5}$. Hence, in step 4 we must process the dicritical points in the order $p_{29}, p_{30}, p_{15}, p_{17}, p_{21}, p_{22}, p_{23}$.

\begin{itemize}
\item Start with $p_{29}$. We have $p' = p_9$ and $p=p_{10}$ because $\frac{m_{p_9}}{n_{p_9}} = \frac{537}{4} < I_{29} = \frac{543}{4} \leqslant \frac{m_{p_{10}}}{n_{p_{10}}} = \frac{544}{4} = 136$. Since it is the first iteration, we take $q_0 = p_{10}$ and perform 4(d).
    \begin{itemize}
    \item $\frac{m_{q_0}}{n_{q_0}} = 136 > \frac{543}{4} = I_{29}$, so that we take $q_1 = p_{11}$, the first satellite of $p_{10}$.
    \item $\frac{m_{q_1}}{n_{q_1}} = \frac{1083}{8} < \frac{543}{4} = I_{29}$, so $q_2 = p_{12}$, the second satellite of $p_{11}$.
    \item $\frac{m_{q_2}}{n_{q_2}} = \frac{407}{3} < \frac{543}{4} = I_{29}$, and therefore $q_3 = p_{13}$, the second satellite of $p_{12}$.
    \end{itemize}
    Since $\frac{m_{q_3}}{n_{q_3}} = \frac{2172}{16} = I_{29}$, this first iteration finishes with $\mathcal{R} = \{ p_{13} \}$ and $\mathcal{S} = \{ O, p_1, \ldots, p_5, p_9, \ldots, p_{13} \}$.
\item Take $p_{30}$. Since $\frac{m_{p_9}}{n_{p_9}} = \frac{537}{4} < I_{30} = \frac{678}{5} \leqslant \frac{m_{p_{14}}}{n_{p_{14}}} = \frac{2712}{20}$, we have $p' = p_9$ and $p = p_{10}$. But $p_{10}$ has already appeared as $p$, and so the rupture point associated to $p_{30}$ is $p_{14}$, the last $p_{10}$-satellite point in $\mathcal{K}(p_{30})$. Up to now we have $\mathcal{R} = \{ p_{13}, p_{14} \}$ and $\mathcal{S} = \{ O, p_1, \ldots, p_5, p_9, \ldots, p_{14} \}$.
\item Take the point $p_{15}$. Since $\frac{m_{p_1}}{n_{p_1}} = 100 < I_{15} = 132 \leqslant \frac{m_{p_2}}{n_{p_2}} = 133$, we have $p' = p_1$ and $p = p_2$. It is the first time $p_2$ appears, so we must perform the iterations of 4(d) starting from $q_0 = p = p_2$:
    \begin{itemize}
    \item $\frac{m_{q_0}}{n_{q_0}} = 133 > 132 = I_{15}$, and hence we take $q_1 = p_3$, the first satellite of $p_2$.
    \item $\frac{m_{q_1}}{n_{q_1}} = \frac{245}{2} < 132 = I_{15}$, and hence we take $q_2 = p_4$, the second satellite of $p_3$.
    \item $\frac{m_{q_2}}{n_{q_2}} = \frac{387}{3} < 132 = I_{15}$, and hence we take $q_3 = p_5$, the second satellite of $p_4$.
    \end{itemize}
    And we stop here because $\frac{m_{q_3}}{n_{q_3}} = \frac{528}{4} = 132 = I_{15}$. We finish this step by setting $\mathcal{R} = \{ p_5, p_{13}, p_{14} \}$ and $\mathcal{S} = \{ O, p_1, \ldots, p_5, p_9, \ldots, p_{14} \}$.
\item The case of $p_{17}$ is exactly the same of $p_{15}$, so we omit it.
\item Take the point $p_{21}$. Since $\frac{m_{p_1}}{n_{p_1}} = 100 < I_{21} = \frac{387}{3} \leqslant \frac{m_{p_2}}{n_{p_2}} = 133$, we have $p' = p_1$ and $p = p_2$. But $p_2$ has already appeared, and hence we obtain that the rupture point associated to $p_{21}$ is $p_4$, the last $p_2$-satellite point in $\mathcal{K}(p_{21})$. Therefore we have by the moment $\mathcal{R} = \{ p_4, p_5, p_{13}, p_{14} \}$ and $\mathcal{S} = \{ O, p_1, \ldots, p_5, p_9, \ldots, p_{14} \}$.
\item Consider the point $p_{22}$. We have again $p' = p_1$ and $p = p_2$ because $\frac{m_{p_1}}{n_{p_1}} = 100 < I_{22} = \frac{799}{7} \leqslant \frac{m_{p_2}}{n_{p_2}}  = 133$, and since $p_2$ has already appeared as the point $p$, the rupture point associated to $p_{22}$ is the last $p_2$-satellite point in $\mathcal{K}(p_{22})$: $p_8$. We finish this step by setting $\mathcal{R} = \{ p_4, p_5, p_8, p_{13}, p_{14} \}$ and $\mathcal{S} = \{ O, p_1, \ldots, p_{14} \}$.
\item We finally take $p_{23}$, the last dicritical point. We have again $p' = p_1$ because $\frac{m_{p_1}}{n_{p_1}} = 100 < I_{23} = \frac{225}{2} \leqslant \frac{m_{p_2}}{n_{p_2}}  = 133$, and hence $p = p_2$. But it has already appeared (three times), and therefore the rupture point associated to $p_{23}$ is $p_7$.
\end{itemize}

\begin{figure}
\begin{center}
\psfrag{A}{$\gamma_1$} \psfrag{B}{$\gamma_2$} \psfrag{C}{$\gamma_3$} \psfrag{D}{$\gamma_4$} \psfrag{E}{$\gamma_5$}

\psfrag{00}{\scriptsize $(50,50)$} \psfrag{01}{\scriptsize $(50,100)$} \psfrag{02}{\scriptsize $(32,132)$} \psfrag{03}{\scriptsize $(13,245)$} \psfrag{04}{\scriptsize $(10,387)$} \psfrag{05}{\scriptsize $(9,528)$} \psfrag{06}{\scriptsize $(3,348)$} \psfrag{07}{\scriptsize $(2,450)$} \psfrag{08}{\scriptsize $(1,799)$} \psfrag{09}{\scriptsize $(9,537)$} \psfrag{10}{\scriptsize $(6,543)$} \psfrag{12}{\scriptsize $(2,1628)$} \psfrag{13}{\scriptsize $(1,2172)$} \psfrag{14}{\scriptsize $(1,2712)$} \psfrag{15}{\scriptsize $(3,1083)$}
\includegraphics[scale=0.22]{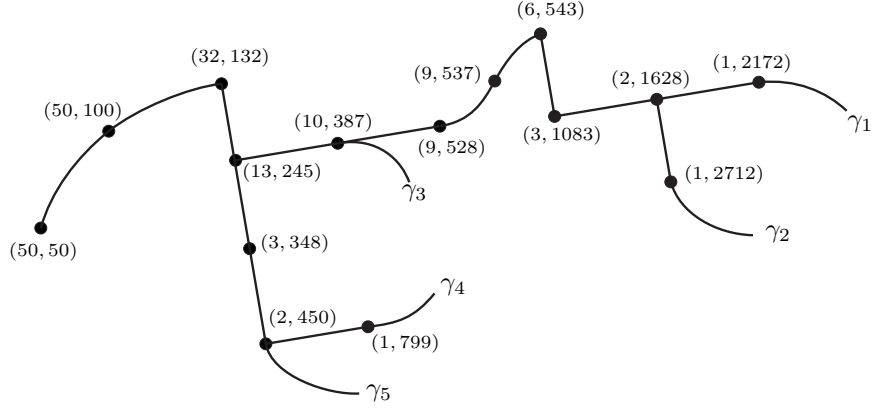}
\end{center}
\caption{\label{fig-ex-C-A} Singular points of $\xi$ with its multiplicities and values $(e_p(\xi),v_p(\xi))$.}
\end{figure}

Thus, the first part of the algorithm finishes with $\mathcal{R} = \{ p_4, p_5, p_7, p_8, p_{13}, p_{14} \}$ and $\mathcal{S} = \{ O, p_1, \ldots, p_{14} \}$, which actually coincide with $\mathcal{R}(\xi)$ and $\mathcal{S}(\xi)$ respectively.

\begin{figure}
\begin{center}
\psfrag{A}{$21$} \psfrag{B}{$43$} \psfrag{C}{$5$} \psfrag{D}{$9$} \psfrag{E}{$5$}

\psfrag{00}{\footnotesize$49$} \psfrag{01}{\footnotesize$49$} \psfrag{02}{\footnotesize$32$} \psfrag{03}{\footnotesize$12$} \psfrag{04}{\footnotesize$9$} \psfrag{05}{\footnotesize$8$} \psfrag{06}{\footnotesize$3$} \psfrag{07}{\footnotesize$2$} \psfrag{08}{\footnotesize$1$} \psfrag{09}{\footnotesize$8$} \psfrag{10}{\footnotesize$6$} \psfrag{12}{\footnotesize$1$} \psfrag{14}{\footnotesize$1$} \psfrag{15}{\footnotesize$1$} \psfrag{16}{\footnotesize$1$} \psfrag{17}{\footnotesize$1$} \psfrag{18}{\footnotesize$1$} \psfrag{19}{\footnotesize$1$} \psfrag{20}{\footnotesize$1$} \psfrag{21}{\footnotesize$1$} \psfrag{22}{\footnotesize$1$}  \psfrag{23}{\footnotesize$1$} \psfrag{24}{\footnotesize$1$} \psfrag{25}{\footnotesize$1$} \psfrag{26}{\footnotesize$1$} \psfrag{27}{\footnotesize$1$} \psfrag{28}{\footnotesize$1$} \psfrag{29}{\footnotesize$1$} \psfrag{30}{\footnotesize$1$} \psfrag{31}{\footnotesize$2$}
\includegraphics[scale=0.22]{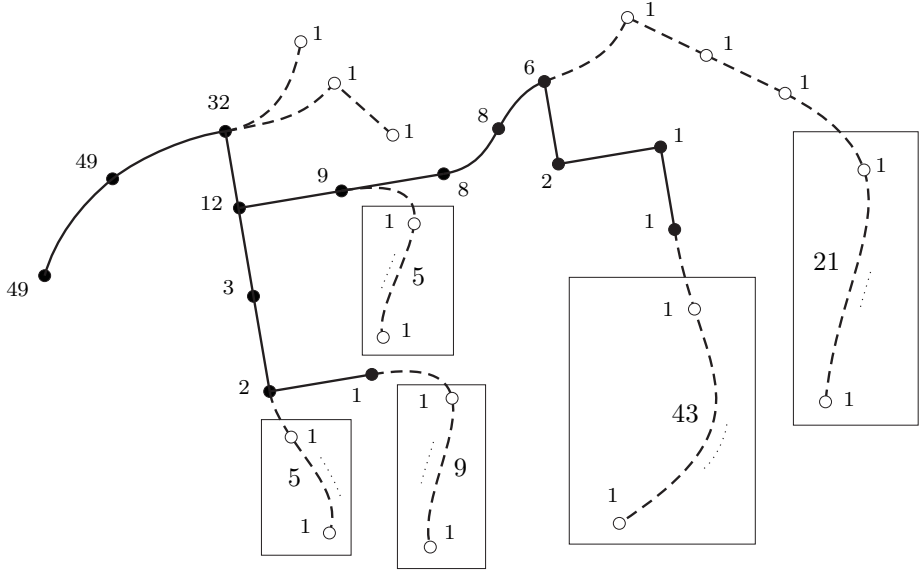}
\end{center}
\caption{\label{fig-ex-C-C} Base points of $\mathcal{J}(\xi)$ with its virtual multiplicities $\nu_p$.}
\end{figure}

The second part begins recovering the values of the rupture points: \[v_{p_4} = 387, \quad v_{p_5} = 528, \quad v_{p_7} = 450, \quad v_{p_8} = 799, \quad v_{p_{13}} = 2172, \quad \text{and} \quad v_{p_{14}} = 2712.\]

Then we take care of the free singular non-rupture points, starting with \[v_O = m_O = 50, \quad v_{p_1} = m_{p_1} = 50 \quad \text{and} \quad v_{p_9} = m_{p_9} = 537\] because $p_1$, $p_2$ and $p_{10}$ are free singular points in the first neighbourhoods of $O$, $p_1$ and $p_9$ respectively. Next, \[v_{p_2} = 132 \quad \text{and} \quad v_{p_{10}} = 543\] because they are the only integers in the intervals \[\left[ \frac{n_{p_2}}{n_{p_5}} m_{p_5}, \frac{n_{p_2}}{n_{p_5}} m_{p_5} + 1 \right) = [132,133) \quad \text{and} \quad \left[ \frac{n_{p_{10}}}{n_{p_{13}}} m_{p_{13}}, \frac{n_{p_{10}}}{n_{p_{13}}} m_{p_{13}} + 1 \right) = [543,544)\] respectively, and $p_5$ (resp. $p_{13}$) is the biggest $p_2$-satellite (resp. $p_{10}$-satellite) rupture point.

Finally, we must consider the satellite non-rupture points, which are $p_3, p_6, p_{11}$ and $p_{12}$. In first place, both $p_3$ and $p_6$ are smaller than $p_5$, the biggest $p_2$-satellite rupture point, and hence we have \[v_{p_3} = m_{p_3} = 245 \quad \text{and} \quad v_{p_6} = m_{p_6} = 348\] because the second instance of step 3 applies. In second place, both $p_{11}$ and $p_{12}$ are smaller than $p_{13}$, which is the biggest $p_{10}$-satellite rupture point. Therefore we get \[v_{p_{11}} = m_{p_{11}} = 1083 \quad \text{and} \quad v_{p_{12}} = m_{p_{12}} = 1628\] by the same reason as above.

\begin{figure}
\begin{center}
\psfrag{A}{$21$} \psfrag{B}{$43$} \psfrag{C}{$5$} \psfrag{D}{$9$} \psfrag{E}{$5$}

\psfrag{00}{\footnotesize$\frac{50}{1}$} \psfrag{01}{\footnotesize$\frac{100}{1}$} \psfrag{02}{\footnotesize$\frac{133}{1}$} \psfrag{03}{\footnotesize$\frac{245}{2}$} \psfrag{04}{\footnotesize$\frac{387}{3}$} \psfrag{05}{\footnotesize$\frac{528}{4}$} \psfrag{06}{\footnotesize$\frac{348}{3}$} \psfrag{07}{\footnotesize$\frac{450}{4}$} \psfrag{08}{\footnotesize$\frac{799}{7}$} \psfrag{09}{\footnotesize$\frac{537}{4}$} \psfrag{10}{\footnotesize$\frac{544}{4}$} \psfrag{12}{\footnotesize$\frac{1628}{12}$}  \psfrag{13}{\footnotesize$\frac{2172}{16}$} \psfrag{14}{\footnotesize$\frac{2712}{20}$} \psfrag{15}{\footnotesize$\frac{135}{1}$} \psfrag{16}{\footnotesize$\frac{135}{1}$} \psfrag{17}{\footnotesize$\frac{268}{2}$} \psfrag{18}{\footnotesize$\frac{389}{3}$} \psfrag{19}{\footnotesize$\frac{801}{7}$} \psfrag{20}{\footnotesize$\frac{452}{4}$} \psfrag{21}{\footnotesize$\frac{397}{3}$} \psfrag{22}{\footnotesize$\frac{817}{7}$} \psfrag{23}{\footnotesize$\frac{460}{4}$} \psfrag{24}{\footnotesize$\frac{546}{4}$} \psfrag{25}{\footnotesize$\frac{1090}{8}$} \psfrag{26}{\footnotesize$\frac{1634}{12}$} \psfrag{27}{\footnotesize$\frac{1636}{12}$} \psfrag{28}{\footnotesize$\frac{2714}{20}$} \psfrag{29}{\footnotesize$\frac{1676}{12}$} \psfrag{30}{\footnotesize$\frac{2798}{20}$} \psfrag{31}{\footnotesize$\frac{1083}{8}$}
\includegraphics[scale=0.22]{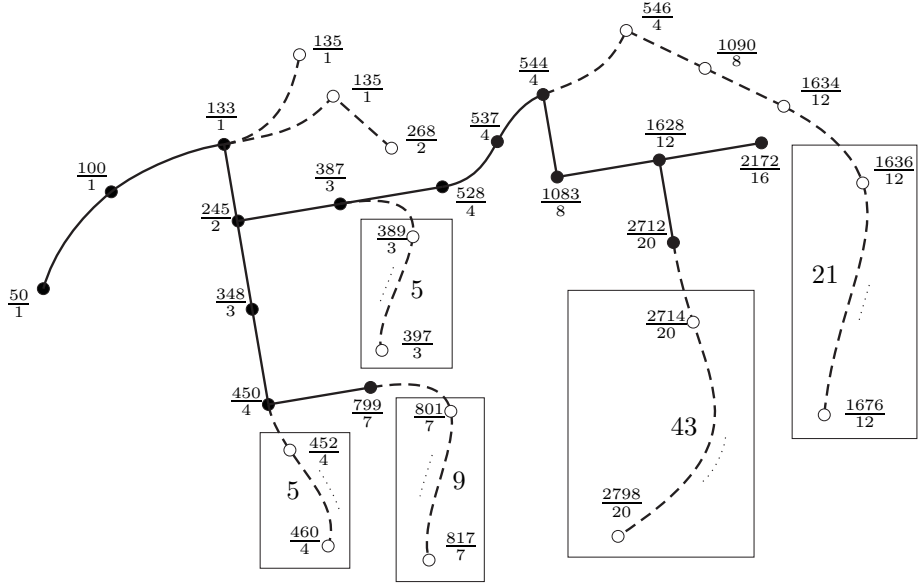}
\end{center}
\caption{\label{fig-ex-C-D} $\mathcal{S}(\xi) \cup BP(\mathcal{J}(\xi))$ with heights of the trunks and multiplicities of $\varphi_p$, $\left(\frac{m_p}{n_p}\right)$.}
\end{figure}

As in all the other examples, it is immediate to check that these values are the values of $\xi$ at its singular points, as claimed.
\end{exam}

\end{document}